\documentclass[letterpaper, 10 pt, journal, twoside]{IEEEtran} 
\IEEEoverridecommandlockouts             

\usepackage{bm}
\usepackage{enumerate}
\usepackage{commath}
\usepackage{graphicx}
\usepackage{amsmath}
\usepackage{subcaption}
\usepackage{breqn}
\usepackage{cite}
\usepackage[table]{xcolor}
\usepackage{booktabs}
\usepackage{empheq}
\usepackage{amsfonts}
\usepackage{amssymb}

\usepackage{amsthm}
\usepackage{hyperref}
\usepackage{xcolor}
\usepackage{mathrsfs}
\usepackage{soul}

\usepackage{accents}
\usepackage{thmtools}
\usepackage{thm-restate}
\usepackage{float}
\usepackage{comment}
\usepackage[normalem]{ulem}
\usepackage{algorithm}
\usepackage{algpseudocode}

\usepackage[normalem]{ulem}

\theoremstyle{plain}

\newtheorem{lemma}{Lemma}

\newtheorem{assumption}{Assumption}

\newtheorem{proposition}{Proposition}

\newtheorem*{problem*}{Problem}
\newtheorem*{theorem*}{Theorem}
\newtheorem{assumption*}{Assumption}

\declaretheorem[name=Theorem]{thm}
\newtheorem{remark}{Remark}

\theoremstyle{definition}

\definecolor{amber}{rgb}{1.0, 0.3, 0.0}

\newcommand{\myvar}[1]{\bm{#1}}

\newcommand{\myset}[1]{\mathcal{#1}} %

\newcommand{\sign}[1]{\textup{sign}{#1}}

\newcounter{MYtempeqncnt}

\begin{document}

\title{\LARGE
  A continuous-time violation-free multi-agent optimization algorithm \\ and its applications to safe distributed control
}

\author{Xiao Tan, Changxin Liu, Karl H. Johansson, and Dimos V. Dimarogonas %
\thanks{ This work was supported in part by Swedish Research Council, ERC CoG LEAFHOUND, EU CANOPIES Project,  Knut and Alice Wallenberg Foundation, and an NSERC Postdoctoral Fellowship. The authors are with the School of EECS, KTH Royal Institute of Technology, 100 44 Stockholm, Sweden (Email: 
        {\tt\small xiaotan, changxin, kallej, dimos@kth.se}).}
}

\maketitle
\thispagestyle{plain}
\pagestyle{plain}

\begin{abstract}

  In this work, we propose a continuous-time distributed optimization algorithm with guaranteed zero coupling constraint violation and apply it to safe distributed control in the presence of multiple control barrier functions (CBF). The optimization problem is defined over a network that collectively minimizes a separable cost function with coupled linear constraints. An equivalent optimization problem with auxiliary decision variables and a  decoupling structure is proposed. A sensitivity analysis demonstrates that the subgradient information can be computed using local information.  This then leads to a subgradient algorithm for updating the auxiliary variables. A case with sparse coupling constraints is further considered, and it is shown to have better memory and communication efficiency. For the specific case of a CBF-induced time-varying quadratic program (QP), an update law is proposed that achieves finite-time convergence. Numerical results involving a static resource allocation problem and a safe coordination problem for a multi-agent system demonstrate the efficiency and effectiveness of our proposed algorithms.

\end{abstract}

\section{Introduction}

Distributed optimization has received increasing attention in the last few decades thanks to increasing computational power availability. Depending on how the  problem is formulated, distributed optimization can be categorized into the \textit{cost-coupled}  and  \textit{constraint-coupled} optimization \cite{notarnicola2019constraint}. In the former case,  the cost function to be minimized is a summation of local cost functions, each of which depends on a common decision variable $\myvar{x}$ subject to a common constraint, that is,
\begin{equation} \label{eq:general_cost_coupled}
    \min_{\myvar{x}\in X} \sum_{i\in \{1,2,...N\}} f_i(\myvar{x}),
\end{equation}
where  the local cost function $f_i$ and global constraint set $X$ are known to agent $i$. One such example is the training process of learning algorithms, where the decision variables are the parameters in the learning model and the local cost function is associated with data accessible only to local computational units. Much attention has been focused on this case in the distributed optimization community; see \cite{nedic2018distributed,yang2019survey} for overviews.

In the constraint-coupled case, each node has its own local decision variable, and the cost function is a summation of local cost functions that depend on the local decision variables only. The  constraints in this case are coupled and involve all or parts of the local variables. The problem setup is then given by 
\begin{equation} \label{eq:general_constraint_coupled}
       \min_{(\myvar{x}_1,...,\myvar{x}_N)\in X} \sum_{i\in \{1,2,...N\}} f_i(\myvar{x}_i),
\end{equation}
where the local cost function $f_i$ and parts of the constraint set $X$ are known to agent $i$. Conceptually, we can refer to the stack of all local variables as the global variable. One example of this setup is the economic dispatch problem in power systems, where the local decision variables are constrained by the shared resource limit, and the size of the global decision variable increases with the size of the power grid.

Albeit its wide outreach in real-time decision-making and profound practical implications, the constraint-coupled optimization problem \eqref{eq:general_constraint_coupled} was less attended to compared to the vast literature on cost-coupled optimization problem \eqref{eq:general_cost_coupled}. One viable approach to \eqref{eq:general_constraint_coupled}
involves solving its dual problem. Assuming that the optimization problem \eqref{eq:general_constraint_coupled} has zero duality gap and the constraint functions are separable, then the dual problem of \eqref{eq:general_constraint_coupled} has a separable objective function with a common decision variable (the dual variable), aligning the problem with the formulation of \eqref{eq:general_cost_coupled}. Thanks to this observation, many distributed optimization algorithms for \eqref{eq:general_cost_coupled} can be readily applied to solve \eqref{eq:general_constraint_coupled} by tackling its dual problem \cite{falsone2017dual,simonetto2016primal,mateos2016distributed,li2020distributed}. 
One major drawback of these duality-based approaches is that, primal convergence is not easily retrieved from dual solutions, known as the primal recovery problem \cite{nedic2009approximate}. The remedy usually involves a running average scheme, which leads to a slow convergence rate. Recent works in \cite{su2021distributed,falsone2023augmented} have proposed various methods to avoid the recovery procedure. Another issue is that, in order to solve for the local primal variable, each node has to solve for a consensus on the dual variable, causing efficiency concerns.

In order to avoid the primal recovery problem as well as enhance efficiency, recently, primal decomposition approach has been pursued as a general distributed solution to the constraint-coupled problem \cite{notarnicola2019constraint}. Instead of looking at the dual problem, primal decomposition approach splits the coupling but separable constraints into local ones using auxiliary variables. The algorithm usually comprises two parts: 1) solving/updating the local primal variables based on local problems, and 2) updating the auxiliary variables. This approach has its roots in large-scale optimization, where the global variable is too large for one computational node to process and  a central coordinating node is required for updating the auxiliary variables. See, for example,  \cite{lasdon2002optimization} for more details. It is until recently that \cite{notarnicola2019constraint} proposes a distributed auxiliary update and provides an asymptotic convergence guarantee. \cite{wang2023distributed} further extends this approach and provides a convergence rate analysis.

One issue that has been largely overlooked in the existing distributed optimization literature is primal feasibility throughout solution iterations. Most algorithms mentioned above or surveyed in \cite{nedic2018distributed,yang2019survey} at best provide an asymptotic convergence guarantee and asymptotic primal feasibility guarantee. This might not be desirable when applying optimization-based algorithms in real-time safe-critical scenarios, i.e., when transient behavior becomes a  crucial safety concern and constraint violation could lead to catastrophic consequences. Conceptually, the problem is challenging since the feasible set is defined over the global decision variable and cannot be evaluated for satisfaction using only local information, which voids the usual practice of projecting the solution onto the feasible set \cite{liu2015second}.

  There are only a few distributed violation-free algorithms for constraint-coupled optimization problems. Recently, \cite{wu2022distributed} proposes a right-hand primal decomposition scheme, yet it only achieves an approximate optimal solution. \cite{abeynanda2023primal} considers primal feasibility in dual decomposition methods, where the solution still only converges to a neighborhood of optimal solution. \cite{tan2021distributed} proposes a violation-free algorithm with optimality guarantee, but only for a single constraint.  A highly relevant work \cite{mestres2023distributed} adopts a similar primal decomposition scheme as in this paper, and uses a cascaded law based on projected saddle-point dynamics and safe gradient flow for updating the auxiliary and primal variables. 
  However, only practical convergence can be established due to a regularization step. The  violation-free  Laplacian-based gradient algorithm in \cite{cherukuri2015distributed}  instead focuses on the special case of economic dispatch problem. In summary, the existing violation-free results either give inexact solutions or are tailored to specific problems. One exception is our recent work  \cite{liu2024achieving} where the focus is on designing a discrete-time, violation-free, distributed and accelerated optimization algorithm. We note that technicality concerns differ significantly depending on whether the dynamics is continuous-time or discrete-time.

In this work, we propose a distributed, continuous-time optimization algorithm for problem \eqref{eq:general_constraint_coupled} with coupling linear constraints. Our contributions are summarized as follows.

\begin{enumerate}
    \item The coupling constraints hold during the solution evolution, and the iterated variable is shown to converge to the optimal solution under a few mild assumptions;
    \item The proposed algorithm is more efficient in memory and communication compared to existing results. Each node only needs to store a variable of size $d_i + M$ and transmit a variable of size $2M$ to its neighbors, where $d_i$ is the size of the local primal variable and $M$ the number of coupling constraints. This is notably smaller compared to, e.g., primal decomposition algorithms \cite{notarnicola2019constraint} and dual decomposition algorithms \cite{falsone2017dual};
    \item We apply the distributed optimization algorithm to the CBF-induced QP as feedback in safe distributed control. To accommodate the slowly time-varying parameters in QP, we further strengthen our proposed algorithm such that finite-time convergence is achieved.

\end{enumerate}

   The organization of the paper goes as follows. We first present  problem formulation in Section~\ref{sec:setup} as well as some leading assumptions. The proposed primal decomposition scheme and the corresponding updates of auxiliary variables are detailed in Section~\ref{sec:equiv_program} and Section~\ref{sec:sensitivity}, respectively. The overall algorithm and the convergence analysis are discussed in Section~\ref{sec:proposed_algorithm}. Two special cases  are further considered in the following sections, focusing on the sparely coupled case and the CBF-induced QP case. Numerical results are shown in Section~\ref{sec:simulation} to demonstrate the effectiveness and the favorable properties of our proposed algorithms. 
   
\section{Preliminaries and problem setup}\label{sec:setup}

    \textit{Notations:}  $\mathbb{R}$ is the set of real numbers, and $\overline{\mathbb{R}} = \mathbb{R}\cup \{-\infty, \infty\}$.  $\mathbb{R}^n$ denotes the $n$ dimensional Euclidean vector space. Each vector $\myvar{a} = (a_1, a_2, ..., a_n)\in \mathbb{R}^n$ is interpreted as a column vector if not stated otherwise. $\textup{diag}(\cdot)$ is a function that only keeps the diagonal entries and sets the off-diagonal entries zero.   $\text{blkdiag}(\myvar{g}_1, \myvar{g}_2, ..., \myvar{g}_n )$ denotes a block diagonal matrix with diagonal blocks $\myvar{g}_1, \myvar{g}_2, ..., \myvar{g}_n$, where $\myvar{g}_i, i = 1,..,n,$ can be either a vector or a matrix. For $x\in \mathbb{R}$, $\displaystyle \sign(x) :=\left\{ \begin{smallmatrix} 1, &x >0; \\
0, &x = 0;\\
-1, &x<0.\end{smallmatrix}\right.$ For any $\myvar{x}= (x_1,..., x_n)\in \mathbb{R}^n$, $\sign(\myvar{x}):=(\sign(x_1), \sign(x_2), ...,\sign(x_n))$.
Vector inequalities are to be interpreted element-wise. $\myvar{0}$ and $ \myvar{1}$ refer to vectors of proper dimensions with all entries to be $0$ or $1$, respectively. 
Given a function $f: \mathbb{R}^n \to \overline{\mathbb{R}}$, $\partial f(\myvar{x})$, the subdifferential of $f$ at $\myvar{x}$, is defined as the set of vectors $\myvar{s}$ such that $f(\myvar{x}^\prime)\geq f(\myvar{x}) + \myvar{s}^\top(\myvar{x}^\prime - \myvar{x})$ for all $\myvar{x}^\prime \in \mathbb{R}^n$. Such a vector $\myvar{s}$ is called a subgradient of $f$ at $\myvar{x}$. When $f(\myvar{x})$ is differentiable,  the gradient of $f$ is denoted $\nabla f$ and the partial derivative $\frac{\partial f}{ \partial x_i}$. For two  matrices $A$ and $B$ of proper dimensions, we use the convention that $\begin{bmatrix}
    A, B
\end{bmatrix}$ denotes the horizontal stacking, and $\begin{bmatrix}
    A; B
\end{bmatrix}$ the vertical stacking. Let $A\in \mathbb{R}^{m\times n}$ be a matrix and $\mathbf{I}\subseteq \{1,2,...,m\}$ be an index set. Then, $A^{\mathbf{I}}$ denotes a matrix composed by the rows indexed in $\mathbf{I}$. Note that $A^{\mathbf{I}} = I^{\mathbf{I}}A$ where $I$ is an identity matrix of proper dimension.

\subsection{Problem setup}
Consider the following multi-agent optimization problem
\begin{equation} \label{eq:centralized_problem}
\begin{split}
    &\min_{\myvar{x}_1,\dots,\myvar{x}_N} \sum_{i\in\myset{I}}  f_i(\myvar{x}_i) \\
    & {s.t.} \quad \left\{
    \begin{array}{cc}
      \sum_{i\in\myset{I}} {\myvar{a}_i^{1 \top}}\myvar{x}_i + b_i^1  \leq 0, \\
        \vdots \\
         \sum_{i\in\myset{I}} {\myvar{a}_i^{m \top}}\myvar{x}_i + b_i^m  \leq 0, \\
         \vdots \\
       \sum_{i\in\myset{I}} {\myvar{a}_i^{M \top}}\myvar{x}_i + b_i^M  \leq 0,
    \end{array}
    \right.
\end{split}
\end{equation}
where $\myset{I} = \{ 1,2,...,N\}$, and there are $M$ linear coupling constraints indexed by $\myset{M} = \{1,2,...,M\}$. Here $\myvar{a}_i^m, \myvar{x}_i\in \mathbb{R}^{d_i}, d_i\in \mathbb{N},  b_i^m\in \mathbb{R}, \ \forall i\in \myset{I}, m\in \myset{M}$. For  brevity, denote $\myvar{x} = [\myvar{x}_{1};\myvar{x}_{2};...; \myvar{x}_{N} ]$, $\myvar{a}^m = [\myvar{a}_1^{m};
 \myvar{a}_2^{m}; ...; \myvar{a}_N^{m}], $ $\myvar{b}^m = (b_1^m, ..., b_N^m)$, $m\in \myset{M}$, $\myvar{a}_i = [\myvar{a}_i^{1};
 \myvar{a}_i^{2}; ...; \myvar{a}_i^{M}]$, $\myvar{b}_i = (b_i^1, ..., b_i^M)$, $i\in \myset{I}$. Only $\myvar{a}_i, \myvar{b}_i, f_{i}$ are known to agent $i$. 

The  $N$ computational nodes communicate over a connected and undirected graph
 $\mathcal{G} = (\myset{I},E)$. $(i,j)\in E$ represents that the agents $i,j$ can communicate with each other. The associated Laplacian matrix \cite{mesbahi2010graph} is denoted $L$, and $\myvar{l}_i$ is its $i$-th row. We make the following assumptions. 

\begin{assumption}\label{ass:convex_f}
For all $i\in \myset{I}$, each function $f_i(\myvar{x}_i)$ is convex.
\end{assumption}

\begin{assumption}\label{ass:finite_cost}
The optimal value of \eqref{eq:centralized_problem} is finite.
\end{assumption}

Stacking $\{ \myvar{a}_{i}^{m}\}_{i\in \myset{I},m\in \myset{M}}$ and  $\{ b_{i}^{m}\}_{i\in \myset{I},m\in \myset{M}}$ by the order of the agents, we obtain $\myvar{a}_{A} = [\myvar{a}_{1};
 \myvar{a}_{2}; ...;\myvar{a}_{N}] $, $\myvar{b}_{A} = [\myvar{b}_{1}; 
 \myvar{b}_{2}; ...; \myvar{b}_{N}] $.
Similarly, stacking $\{ \myvar{a}_{i}^{m}\}_{i\in \myset{I},m\in \myset{M}}$ and  $\{ b_{i}^{m}\}_{i\in \myset{I},m\in \myset{M}}$ by the order of  the constraints, we obtain $\myvar{a}_{C} = [\myvar{a}^{1}; 
 \myvar{a}^{2}; ...; \myvar{a}^{M}] $, $\myvar{b}_{C} = [\myvar{b}^{1}; 
 \myvar{b}^{2}; ...; \myvar{b}^{M}] $. 
 There exists a constant permutation matrix\footnote{A permutation matrix $J$ is a square binary matrix that has exactly one entry of $1$ in each row and each column and zeros elsewhere. A permutation matrix is orthogonal, i.e., $J J^\top = J^\top J =  I$. } $J$ such that $\myvar{b}_C = J \myvar{b}_A$. Note that $J$ only depends on $ N$ and $ M$ and is independent of the communication graph $\mathcal{G}.$

In this work, we aim to find a continuous-time distributed algorithm that solves \eqref{eq:centralized_problem} while satisfying all coupling constraints for all time. 
In \cite{tan2021distributed}, a special case of \eqref{eq:centralized_problem} with quadratic cost and a single coupling constraint is investigated.
However, the approach in \cite{tan2021distributed} is difficult to generalize to convex programs with multiple coupling constraints, as studied in this work.

Our proposed approach is based on two key observations. 1) By introducing a set of auxiliary variables as decision variables, an equivalent optimization problem can be constructed with separable costs and constraints. 2) When viewing these auxiliary variables as independent variables, the subgradient of the optimal value with respect to these variables can be computed using  local information. The first fact motivates the primal decomposition scheme. The second fact motivates the distributed updated law for these variables leveraging subgradient algorithms. The convergence and optimality of this approach are later shown using tools from nonsmooth analysis. Two special cases with sparse coupling constraints or quadratic cost functions are also discussed in the sequel.

\section{Equivalent Problem} \label{sec:equiv_program}
In this section we present a new  optimization problem that is equivalent to \eqref{eq:centralized_problem} but with a separable structure. Similar decomposition schemes were also discussed in \cite[Proposition 1]{tan2021distributed} and \cite[Proposition 4.1]{mestres2023distributed}. Introduce auxiliary decision variables $\{ y_i^m \}_{i\in \myset{I}, m\in \myset{M}}$ where $y_i^m \in \mathbb{R}$. Denote $\myvar{y}^{m} = (y^m_1, ..., y_N^m) \in \mathbb{R}^N, m\in \myset{M}$, and $\myvar{y} = [\myvar{y}^{1};\myvar{y}^{2};...; \myvar{y}^{M} ] $. Consider
\begin{equation}
\label{eq:reformulated_problem}
\begin{split}
    &\min_{\myvar{x}, \myvar{y}^1,\dots,\myvar{y}^M} \sum_{i\in\myset{I}} f_i(\myvar{x}_i)   \\
    & {s.t.} \quad \left\{
    \begin{array}{cc}
          {A^1}^{\top}\myvar{x} + L \myvar{y}^1 + \myvar{b}^1 \leq \myvar{0}  \\
         \vdots \\
          {A^M}^{\top}\myvar{x} + L\myvar{y}^M + \myvar{b}^M \leq \myvar{0}
    \end{array}
    \right.
\end{split}
\end{equation}
where recall that $L$ is the graph Laplacian matrix, $A^m = {\rm blkdiag}(\myvar{a}_1^m, \dots, \myvar{a}^m_N) \in \mathbb{R}^{(\sum_i d_i)\times N}$.

When organizing the constraints by the order of agents, the problem can be rewritten as 
\begin{equation}
\label{eq:reformulated_problem_agent} 
\begin{split}
        &\min_{\myvar{x}, \myvar{y}^1,\dots,\myvar{y}^M} \sum_{i\in \myset{I}} f_i(\myvar{x}_i)  \\
    & {s.t.} \quad \left\{
    \begin{array}{cc}
          {A}_1^{\top}\myvar{x}_1 + (I_M \otimes \myvar{l}_1) \myvar{y}  + \myvar{b}_1 \leq \myvar{0}  \\
         \vdots \\
         {A}_i^{\top}\myvar{x}_i + (I_M \otimes \myvar{l}_i) \myvar{y}  + \myvar{b}_i \leq \myvar{0} \\
         \vdots \\
          {A}_N^{\top}\myvar{x}_N + (I_M \otimes \myvar{l}_N) \myvar{y}  + \myvar{b}_N \leq \myvar{0}
    \end{array}
    \right. 
\end{split}
\end{equation}
where recall that $\myvar{l}_i$ is the $i$th row of the graph Laplacian, $A_i = \begin{bmatrix}
    \myvar{a}_i^{1}, \myvar{a}_i^{2}, ...,  \myvar{a}_i^M
\end{bmatrix} \in \mathbb{R}^{d_i\times M}$.
 One verifies that the optimization problems in \eqref{eq:reformulated_problem}, 
\eqref{eq:reformulated_problem_agent} are the same up to a permutation of constraints. An illustration of the decomposition approach in shown in Fig.~\ref{fig:decomposition}. 

\begin{figure}
    \centering
    \includegraphics[width=0.9\linewidth]{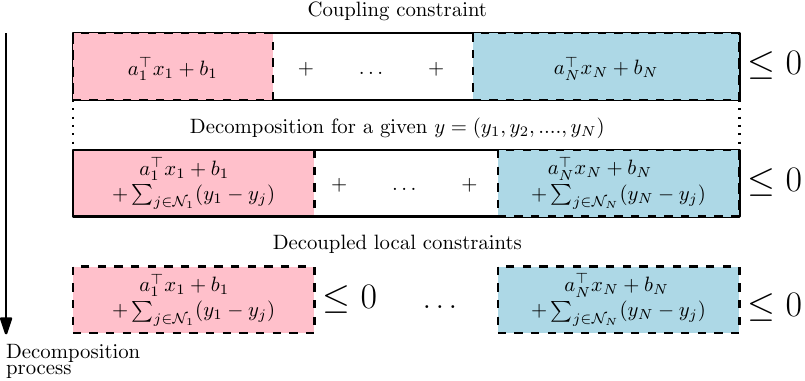}
    \caption{Illustration of the constraint decomposition scheme. Here only one coupling constraint is considered and the superscript is neglected for simplicity. By introducing an auxiliary variable $\myvar{y}$, one coupling constraint becomes $N$ local constraints. Proposition~\ref{prop:equivalent_problem} shows that when $\myvar{y}$ is a decision variable, the two class of constraints represent the same feasible region in $\myvar{x}$. Our proposed algorithm locally updates $y_i$ to obtain an ``optimal" constraint decomposition in the sense that $\myvar{a}_i^\top \myvar{x}_i^\star + b_i + \sum_{j\in \mathcal{N}_i} (y_i - y_j) \leq 0$ for all $i\in \mathcal{I}$, where $\myvar{x}_i^\star$ is the optimal solution to the original optimization problem.    }
    \label{fig:decomposition}
\end{figure}

 \begin{proposition}[Equivalence] \label{prop:equivalent_problem}
    The two optimization problems in \eqref{eq:centralized_problem} and \eqref{eq:reformulated_problem} are equivalent, in the sense that
    \begin{enumerate}
        \item for any feasible solution   $(\myvar{x}^{\prime}, \myvar{y}^{\prime}  )$ to \eqref{eq:reformulated_problem}, $\myvar{x}^{\prime}$ is a feasible solution to \eqref{eq:centralized_problem};
        \item for any feasible solution $\myvar{x}^{\prime}$ to \eqref{eq:centralized_problem}, there exists a  $\myvar{y}^{\prime} \in \mathbb{R}^{NM}$ such that $(\myvar{x}^{\prime}, \myvar{y}^{\prime}  )$ is a feasible solution to \eqref{eq:reformulated_problem};
        \item the two problems have the same cost function.
    \end{enumerate}
\end{proposition}
\begin{proof}
    See Appendix.
\end{proof}

If $\myvar{y}^m, m\in \myset{M},$ are viewed as independent variables instead of decision variables, then the problem \eqref{eq:reformulated_problem_agent} has a natural separable structure for a given $\myvar{y}$. In the following, we consider a distributed solution to \eqref{eq:centralized_problem}, composed of the local optimization problems for agent $ i\in \myset{I}$
\begin{equation} \label{eq:local_problem}
\begin{split}
    &\hspace{3.5cm} \min_{\myvar{x}_i} f_i(\myvar{x}_i) \\
    & {s.t.} \quad \left\{
    \begin{array}{cc}
    {\myvar{a}_i^1}^{\top}\myvar{x}_i + \sum_{j\in\mathcal{N}_i}(y_i^1-y_j^1)+ b_i^1  \leq 0 \\
        \vdots
       \\
{\myvar{a}_i^M}^{\top}\myvar{x}_i + \sum_{j\in\mathcal{N}_i}(y_i^M-y_j^M) + b_i^M  \leq 0
    \end{array}
    \right.
\end{split}
\end{equation}
 and  some local update law for the variable $y_i^m, i\in \myset{I}, m\in \myset{M},$ that will be detailed in the next section. We note that this local optimization problem will give an optimal $\myvar{x}_i$ if and only if $\myvar{y}^m, m\in \myset{M}$ is optimal to the problem in \eqref{eq:reformulated_problem} in view of Proposition~\ref{prop:equivalent_problem}. Nevertheless, for any given $\myvar{y}$, we have the following result.

\begin{proposition}[Constraint satisfaction] \label{prop:constraint_satisfaction}
    For a given $\myvar{y}$, if $\myvar{x}_i$, $ i\in\myset{I}$, is feasible to the local optimization problem in \eqref{eq:local_problem}, then the stacked vector $[\myvar{x}_1; \myvar{x}_2; ...; \myvar{x}_N]$ is a feasible solution to \eqref{eq:centralized_problem}.
\end{proposition}
\begin{proof}
  Consider the $m$-th constraint in \eqref{eq:centralized_problem}. By summing up the $m$-th constraint of \eqref{eq:local_problem} for each agent, we have $\sum_{i\in \myset{I}} \myvar{a}_i^{m \top} \myvar{x}_i + \sum_{i\in \myset{I} }\myvar{l}_i \myvar{y}^m + \sum_{i\in \myset{I}} b_i^m =\sum_{i\in \myset{I}} \myvar{a}_i^{m \top} \myvar{x}_i + b_i^m \leq 0$, considering  that $\sum_{i\in \myset{I} }\myvar{l}_i \myvar{y}^m = \myvar{1}^\top L \myvar{y}^m = 0$.
\end{proof}

\section{Sensitivity analysis of local and global optimization problems} \label{sec:sensitivity}
In this section, we will investigate how $\myvar{y}$ affects the optimal value of the local optimization problems. We will propose an updating scheme for $\myvar{y}$  such that the sum of the optimal values of the local problems always decreases until reaching the optimal value of the centralized problem.

 \begin{assumption}\label{ass:local_slater}
     For any $\myvar{y}\in \mathbb{R}^{NM}$, Slater's conditions hold for all local optimization problems, i.e., there exists $\myvar{x}_i$ that is strictly feasible to \eqref{eq:local_problem}.
 \end{assumption}
One sufficient condition for this assumption to hold is that $\text{Rank}(A_i) = M, \forall i\in \myset{I}$. In the following we assume Assumption~\ref{ass:local_slater} holds.  For a given $\myvar{y}$, consider the local problem \eqref{eq:local_problem} 
 \begin{equation} \label{eq:local_QP_phi}
   \begin{aligned}
    \phi_i(\myvar{y}) &:=    \min  f_i(\myvar{x}_i)  \\
    & s.t. \ \ {A}_i^\top \myvar{x}_i + ({I}_M \otimes \myvar{l}_i)\myvar{y} + \myvar{b}_i\leq 0.
    \end{aligned} 
\end{equation}

Under Assumption~\ref{ass:local_slater}, $\phi_i(\myvar{y})$ is well-defined for all $\myvar{y}\in \mathbb{R}^{NM}$. Given $\myvar{y}\in \mathbb{R}^{NM}$, define 
\begin{equation} \label{eq:phi_y}
    \phi(\myvar{y}) = \sum_{i\in \myset{I}} \phi_i(\myvar{y})
\end{equation} 
In general, $\phi(\myvar{y})$ is not differentiable. In the following, we use tools from  nonsmooth analysis \cite{clarke2008nonsmooth}. Let the subdifferential, i.e., the set of all subgradients, of $\phi(\myvar{y})$ be $\partial \phi(\myvar{y})$. We have the following results characterizing the subdifferential of  $\phi(\myvar{y})$. 

\begin{proposition}[Subdifferential] \label{prop:subdifferential}
    Let Assumptions~\ref{ass:convex_f},~\ref{ass:finite_cost}, and~\ref{ass:local_slater} hold.  Denote by $\myvar{c}_i = (c_i^1, ..., c_i^M)\in \mathbb{R}^M$ a Lagrange multiplier to \eqref{eq:local_problem} for $i\in \myset{I}$. Define $\zeta_i^m = \sum_{j\in N_i} (c_i^m - c_j^m)$, $\myvar{\zeta}^m = (\zeta_1^m, \zeta_2^m,..., \zeta_N^m), \myvar{\zeta} = \begin{bmatrix}
    \myvar{\zeta}^1;  \myvar{\zeta}^2;  ...;  \myvar{\zeta}^M
\end{bmatrix}$. Then 
\begin{enumerate}
    \item $\myvar{c}_i\in \mathbb{R}^M$, not necessarily unique, always exists,
    \item $\phi(\cdot)$ is convex, but not strongly convex,
    \item $\myvar{\zeta} \in \partial \phi(\myvar{y}).$
\end{enumerate}
   
\end{proposition}

\begin{proof}

 From Assumptions~\ref{ass:convex_f},~\ref{ass:finite_cost}, and~\ref{ass:local_slater},  the strong duality theorem for convex inequality constraints \cite[Proposition 3.5.1]{bertsekas1995nonlinear} applies, and thus, the strong duality holds and at least one Lagrange multiplier exists. This proves Point 1).

We show the convexity of $\phi_i(\myvar{y})$ by definition. For any $\myvar{y}^\prime, \myvar{y}^{\prime\prime} $, denote $\myvar{x}^\prime,\myvar{x}^{\prime\prime}$ the corresponding optimal solution of \eqref{eq:local_QP_phi}. It is clear that $\phi_i(\lambda \myvar{y}^\prime + (1-\lambda) \myvar{y}^{\prime\prime} ) \leq f_i (\lambda \myvar{x}^\prime + (1-\lambda) \myvar{x}^{\prime\prime} ) \leq \lambda f_i(\myvar{x}^\prime) + (1-\lambda)f_i(\myvar{x}^{\prime\prime}) = \lambda \phi_i(\myvar{y}^\prime) + (1-\lambda)\phi_i(\myvar{y}^{\prime\prime})  $, where the inequalities hold thanks to the linearity of the constraints and the convexity of $f_i(\cdot)$, respectively. Thus, $\phi(\myvar{y})$ is also convex. $\phi(y)$ not being strongly convex is evident since $\phi(y + s\mathbf{1}) = \phi(y)$ for any $s\in \mathbb{R}$.

  Now consider a perturbed version of \eqref{eq:local_QP_phi} with $\myvar{y}^\prime$ and denote its optimal value $\phi_i(\myvar{y}^\prime)$. From sensitivity analysis \cite[Equation 5.57]{boyd2004convex}, we know 
 \begin{equation} \label{eq:sensitivity}
      \phi_i(\myvar{y}^\prime)  \geq   \phi_i(\myvar{y}) + \myvar{c}_i^\top ({I}_M \otimes \myvar{l}_i)( \myvar{y}^\prime - \myvar{y})
 \end{equation}
for all $\myvar{y}^\prime$. This means $ ( {I}_M \otimes \myvar{l}_i^\top) \myvar{c}_i$ is a subgradient of $\phi_i(\myvar{y})$. Since $\phi(\myvar{y})$  is convex,  $\partial \phi$ is non-empty and takes compact and convex values \cite[Proposition 9]{cortes2008discontinuous}. Based on the calculus of subdifferentials: $\partial(f_1 + f_2 )(\myvar{x}) \supseteq \partial f_1(\myvar{x}) + \partial f_2(\myvar{x})$, we know 
\begin{equation} 
    \sum_{i\in \myset{I}}  ( {I}_M \otimes \myvar{l}_i^\top) \myvar{c}_i \in \partial \phi(\myvar{y}).
\end{equation}
Note that
$$ ( {I}_M \otimes \myvar{l}_i^\top) \myvar{c}_i = \begin{bmatrix}
    \myvar{l}_i^\top & 0 &  & 0 \\
    0 & \myvar{l}_i^\top &  & 0 \\
    0 & 0 & \ddots & 0 \\
    0 & 0 &  & \myvar{l}_i^\top 
\end{bmatrix} \begin{bmatrix}
    c_i^1\\
    c_i^2 \\
    ...\\
    c_i^M
\end{bmatrix} =  \begin{bmatrix}
    c_i^1 \myvar{l}_i^\top \\
    c_i^2  \myvar{l}_i^\top\\
    ...\\
    c_i^M \myvar{l}_i^\top
\end{bmatrix}
$$
and that 
$$\sum_{i\in \myset{I}} c_i^m \myvar{l}_i^\top = \begin{bmatrix}
    \myvar{l}_1^\top & \myvar{l}_2^\top & ... & \myvar{l}_N^\top
    \end{bmatrix} \myvar{c}^m =  L^\top \myvar{c}^m = L \myvar{c}^m $$
 with $\myvar{c}^m = (c_1^m, c_2^m, ..., c_N^m)$. We thus obtain
\begin{multline}  \label{eq:global_subgradient} 
    \sum_{i\in \myset{I}}  ( {I}_M \otimes \myvar{l}_i^\top) \myvar{c}_i  = [
    L \myvar{c}^1; L \myvar{c}^2;  ...; L \myvar{c}^M
    ]   \\ = ( {I}_M \otimes L) \myvar{c}_C \in  \partial \phi(\myvar{y}),
\end{multline}  
where $\myvar{c}_C =[\myvar{c}^1; \myvar{c}^2; ...; \myvar{c}^M] $ (ordered by constraints). Equivalently, we have $\myvar{\zeta} \in \partial \phi(\myvar{y})$. This proves Point 3).
\end{proof}

\begin{remark}
    Recall that $\myvar{c}_i$ can be calculated solely based on the local problem at node $i$. Proposition~\ref{prop:subdifferential} reveals that although each computational node does not have the full information about the subgradient of $\phi(\myvar{y})$, each computational node has access to the subgradient of $\phi(\myvar{y})$ with respect to the local auxiliary variable $\myvar{y}_i$ by communicating local Lagrange multipliers to its neighbors. Specifically, we have that
\begin{equation}
    \frac{\partial \phi(\myvar{y})}{\partial y_i^m} = \zeta_i^m = \sum_{j\in N_i}(c_i^m - c_j^m), \forall m\in \myset{M}
\end{equation}
whenever $\phi(\myvar{y})$ is differentiable.
\end{remark}

\section{Proposed algorithm}\label{sec:proposed_algorithm}

In this section, we give details about the proposed continuous-time distributed optimization algorithm. A pseudo-code for computational node $i$ is given in Algorithm~\ref{alg:distributed_optimization}. Note that since the analysis is done in continuous time, we are assuming that the four steps in the loop block are executed at the same time instant, i.e.,  $\myvar{y}_i, \myvar{y}_j, \myvar{c}_i, \myvar{c}_j$ in Steps 4 and 6 are computed and transmitted simultaneously. For digital implementation, all agents synchronously communicate $\myvar{y}_i$ and $\myvar{c}_i$ in Steps 4 and 6. We also choose a small sampling time to approximate the continuous-time update of $\myvar{y}_i$ in Step 7.

\begin{algorithm} 
\caption{Distributed Optimization Algorithm}
\label{alg:distributed_optimization}
\begin{algorithmic}[1]
\State \textbf{Locally stored state:}  $\myvar{x}_i, \myvar{y}_i$
\State  \textbf{Initialization:}  choose arbitrary $\myvar{y}_i$.
\Loop
\State Send $\myvar{y}_i$ to and gather $\myvar{y}_j$ from neighbors $j\in N_i$
\State Compute $(\myvar{x}_i, \myvar{c}_i)$ as the primal-dual optimizer of \eqref{eq:local_problem}
\State Send $ \myvar{c}_i$ to and gather $\myvar{c}_j$ from  neighbors $j\in N_i$
\State Update $\myvar{y}_i$ according to \eqref{eq:adaptive law}
\EndLoop
\end{algorithmic}
\end{algorithm}

\begin{thm} \label{thm:convergence and constraints}
    Let Assumptions~\ref{ass:convex_f},~\ref{ass:finite_cost} and~\ref{ass:local_slater} hold.   Denote by $\myvar{c}_i(\myvar{y})$  the Lagrange multiplier of  \eqref{eq:local_problem}  given $\myvar{y}$. If $\myvar{y}_i$ is updated according to 
\begin{equation} \label{eq:adaptive law}
    \dot{\myvar{y}}_i = - k_0 \sum_{j\in N_i} (\myvar{c}_i(\myvar{y}) - \myvar{c}_j(\myvar{y})) \text{ for almost all } t,
\end{equation}
where $k_0>0$ is a constant gain and a Caratheodory solution\footnote{A Caratheodory solution of a differential equation $\dot{\myvar{x}}(t)= X(\myvar{x}(t))$ or a differential inclusion $\dot{\myvar{x}}(t) \in X(\myvar{x}(t)) $ is an absolutely continuous function $\myvar{x}(t):[0,\infty) \to \mathbb{R}^n$  that satisfies the differential equation or the differential inclusion for almost all $t$, respectively \cite{cortes2008discontinuous}. } exists, then 
\begin{enumerate}
    \item the Caratheodory solution is unique and converges to the set of optimal $\myvar{y}$ of the equivalent problem \eqref{eq:reformulated_problem},
    \item the solution to the local problem \eqref{eq:local_problem}  asymptotically converges to (one of) the optimal solution(s) of the centralized problem in \eqref{eq:centralized_problem},
    \item and the coupling constraints in \eqref{eq:centralized_problem} hold during the solution evolution.
\end{enumerate}
\end{thm}

\begin{proof}
From Proposition~\ref{prop:subdifferential}, $\myvar{c}_i$ is well-defined for all $\myvar{y}$, but it is not necessarily unique. From the right-hand side of \eqref{eq:adaptive law}, one derives
    \begin{equation}
         \dot{\myvar{y}}_i = -k_0\begin{bmatrix}
    \sum_{j\in N_i} (c^1_i - c^1_j) \\
    \sum_{j\in N_i} (c^2_i - c^2_j) \\
    ... \\
    \sum_{j\in N_i} (c^M_i - c^M_j)
    \end{bmatrix} = - k_0\begin{bmatrix}
    l_i \myvar{c}^1 \\
     l_i \myvar{c}^2 \\
    ... \\
     l_i \myvar{c}^M
    \end{bmatrix}
    \end{equation}
    for almost all $t$. Re-order  $\{ -k_0 l_i \myvar{c}^m \}_{i\in \myset{I}, m \in \myset{M}}$  by the order of constraints, we obtain $\dot{\myvar{y}} = \begin{bmatrix}
       \dot{\myvar{y}}^1; \dot{\myvar{y}}^2; ...; \dot{\myvar{y}}^M 
    \end{bmatrix}$ where $$ \dot{\myvar{y}}^m =  - k_0[
    l_1 \myvar{c}^m;  \\    l_2 \myvar{c}^m; ...;
     l_N \myvar{c}^m
    ] = - k_0  L\myvar{c}^m.$$     Thus, $\dot{\myvar{y}} =  - k_0  (I_M \otimes L)\myvar{c}_C$ for almost all $t$, where we recall $\myvar{c}_C =[\myvar{c}^1; \myvar{c}^2; ...; \myvar{c}^M] $ (ordered by constraints). Denote the  Caratheodory solution  of this differential equation to be $\myvar{y}(t)$.

  Now consider Caratheodory solutions of the differential inclusion $\dot{\myvar{y}}  \in  - k_0 \partial \phi(\myvar{y}) = -  \partial k_0 \phi(\myvar{y}) $  in view of Proposition~\ref{prop:subdifferential},  where $\phi(\myvar{y})$ is defined in \eqref{eq:phi_y}. From \cite[Gradient Differential Inclusion of a Convex Function]{cortes2008discontinuous}, we know such a solution exists and is unique. Recall that  $\dot{\myvar{y}}(t)\in -\partial k_0 \phi(\myvar{y}(t))$ for almost all $t$, thus  $\myvar{y}(t)$ is  also unique.

  In order to investigate how $\phi(\myvar{y}(t))$ evolves with respect to time, we check the  set-valued lower Lie derivative along the differential inclusion  \begin{multline}
       \underline{L}_{-\partial k_0 \phi(\myvar{y})} \phi(\myvar{y})  = \{ a\in \mathbb{R}:  \text{there exists } \myvar{\xi}\in \partial \phi(\myvar{y})  \\\text{ such that }   a = \min_{\myvar{v}\in -\partial k_0 \phi(\myvar{y})} \myvar{\xi}^\top \myvar{v} \}
  \end{multline}

For any $\myvar{\xi}\in \partial \phi(\myvar{y}) $, there exists $\myvar{v} = 
 -k_0 \myvar{\xi} \in -\partial k_0 \phi(\myvar{y})$ such that $\myvar{\xi}^\top \myvar{v}= -k_0 \| \myvar{\xi}  \|^2\leq 0$. Hence $\sup  \underline{L}_{-\partial k_0 \phi(\myvar{y})} \phi(\myvar{y}) \leq 0$ for all $\myvar{y}$, and the equality holds if and only if $\myvar{0}\in \partial \phi(\myvar{y})$. Following \cite[Propositions 13]{cortes2008discontinuous} and the uniqueness of the solutions, we know $\phi(\myvar{y}(t))$ is nonincreasing. In particular, when $\myvar{0}\notin \partial \phi(\myvar{y})$, $\sup  \underline{L}_{-\partial k_0 \phi(\myvar{y})} \phi(\myvar{y}) < 0$, and thus $\phi(\myvar{y}(t))$ is strictly decreasing.

 Denote $h(t) = \sup  \underline{L}_{-\partial k_0 \phi(\myvar{y}(t))} \phi(\myvar{y}(t))$, $\alpha = \lim\sup_{t\to \infty} h(t)$. Such a limit exists and is bounded since $h(t)$ is upper bounded by $0$. Assume $\alpha <0$, then, $\exists T>0 $ and $\epsilon = -\alpha /2$ such that $h(t)< \alpha  + \epsilon = \alpha /2<0$ for all $t\geq T$. This implies that $\phi(\myvar{y}(t)) < \phi(\myvar{y}(T)) + \alpha (t-T)/2 $ for all $t\geq T$, which contradicts with Assumption~\ref{ass:finite_cost}. Thus, we have $\limsup_{t\to \infty} h(t) = 0$. This means that a subsequence of $h(t)$ converges to $0$ asymptotically, which leads to the conclusion that a subsequence of $\myvar{y}(t)$ converges to the set $ \{\myvar{y}:  \myvar{0} \in \partial\phi(\myvar{y}) \}$. This proves Point 1).
 
 Since $\myvar{y}(t)$ converges to the optimal set $\{\myvar{y}:  \myvar{0} \in \partial\phi(\myvar{y}) \}$ asymptotically, and in view of the equivalence between \eqref{eq:centralized_problem} and \eqref{eq:reformulated_problem} from Proposition~\ref{prop:equivalent_problem}, we know the optimal solution to \eqref{eq:centralized_problem} is obtained asymptotically. This shows Point 2). Point 3) follows from Assumption~\ref{ass:local_slater} and Proposition~\ref{prop:constraint_satisfaction}. 
 \end{proof}

\begin{remark}\label{rem:convergence}
Convergence of the subgradient differential inclusion of a convex function was investigated in  \cite{clarke2008nonsmooth} and \cite[Gradient Differential Inclusion of a Convex Function]{cortes2008discontinuous}, where the convergence property has been established. However, we note that here $\phi(\myvar{y})$ has infinitely many local minima and they may not be bounded. Proofs in \cite{clarke2008nonsmooth} and \cite{cortes2008discontinuous} rely on nonsmooth Lyapunov function analysis (with an assumption on the uniqueness of the local minima)  or LaSalle's principle (with an assumption on the boundedness of the local minima), both of which are not applicable here. Instead, we did not prove the convergence of $\myvar{y}(t)$ to a particular optimal point, but to an optimal set $\{\myvar{y}:  \myvar{0} \in \partial\phi(\myvar{y}) \}$.
\end{remark}

\begin{remark}[Caratheodory solution]
One may wonder why a Caratheodory solution is considered in Theorem~\ref{thm:convergence and constraints} and how to guarantee/verify its existence. In Proposition~\ref{prop:subdifferential}, we have established that, for a given $\myvar{y}$, the Lagrange multiplier $c_i$ always exists but may not be unique under  Assumptions~\ref{ass:convex_f},~\ref{ass:finite_cost}, and~\ref{ass:local_slater}. If  $c_i(\myvar{y})$, as a function of $\myvar{y}$, is continuous, then classic solutions to the ODE in~\eqref{eq:adaptive law} exist, and of course a Caratheodory solution exists. We introduce the Caratheodory solution so that we can neglect discontinuities of $c_i(\myvar{y}(t))$ for a duration of measure zero. Theoretically,  for the local optimization problem~\eqref{eq:local_QP_phi} parameterized by $\myvar{y}$, when the cost function is smooth and the Strong Second Order Sufficient Condition (SSOSC)  holds at $\myvar{y}^\prime$, there exists a neighborhood of $\myvar{y}^\prime$ such that a unique primal-dual optimizer $(\myvar{x}_i(\myvar{y}), \myvar{c}_i(\myvar{y}) )$, when viewed as a function of $\myvar{y}$, exists and is locally Lipschitz continuous in $\myvar{y}$ \cite{robinson1980strongly}. Other conditions exist for ensuring continuous primal-dual pairs of perturbed optimization problems. See a recent survey \cite{mestres2023robinson}  and references therein. Numerically, if there are multiple eligible local Lagrange multipliers, one can choose $\myvar{c}_i $ that is close to $\myvar{c}_i^{-}$, where $\myvar{c}_i^{-}$ denotes the variable $\myvar{c}_i$ calculated one step before.

\end{remark}

\begin{remark}[Differentiability] \label{rem:differentiability}
     \cite[Corollary 7.3.1]{florenzano2001finite} states that the perturbed optimal value function is locally differentiable if and only if the Lagrange multiplier is unique. Thus, one sufficient condition for local differentiability of $\phi_i(\myvar{y})$ is  that the optimization problem in \eqref{eq:local_QP_phi} fulfills the Linearly Independent Constraint Qualification (LICQ), which ensures the existence and uniqueness of the Lagrange multiplier \cite{wachsmuth2013licq}. When $\phi_i(\myvar{y})$'s are differentiable, the subgradient then becomes the usual gradient, i.e., $\partial \phi_i(\myvar{y}) = \{ \nabla \phi_i(\myvar{y})\}$. In this case, the adaptive law in \eqref{eq:adaptive law} becomes the well-established gradient flow. This will be discussed in detail for one special case, namely the control barrier function-induced quadratic programs, in Section~\ref{sec:qp_case}.

\end{remark}

\begin{remark}[Memory, computation, and communication efficiency]

In Algorithm~\ref{alg:distributed_optimization}, each computational node $i$ stores variables $\myvar{x}_i$ and $\myvar{y}_i$ of dimension $d_i + M$, while broadcasting $\myvar{y}_i$ and $\myvar{c}_i$ of dimension $2M$ to neighbors. Compared to the primal decomposition scheme in \cite{notarnicola2019constraint}, our approach requires smaller storage ($d_i + M$ versus  $d_i + D_i M$) and communication ($2M$ versus  $2 D_i M$) variables, where $D_i$ is the degree of the communication graph at node $i$ \cite{mesbahi2010graph}. \cite{mestres2023distributed} adopts the same decomposition scheme as in this work, but requires each local node to store $2d_i + 2M$ variables. Additionally, the proposed algorithm requires less computation for local nodes at each iteration, as they only need to solve a single optimization problem with the size of decision variable $d_i$, unlike the dual decomposition scheme in \cite{falsone2017dual}, which involves solving two optimizations with sizes $d_i$ and $M$.

\end{remark}

\section{A sparse case}\label{sec:sparse_case}

What has been discussed so far deals with the most general case, and is applicable in a densely coupled scenario, i.e.,  every computational node is involved in every coupling constraint. In many cases, the constraint coupling is sparse and, usually, is consistent with the communication topology. For example, in multi-robot applications \cite{Wang2017a,fernandez2023distributed,mestres2024distributed}, collision avoidance is a pair-wise constraint between two robots, and they need to have some information exchange between them in order to work out a safe path. Exploiting the sparsity of the constraint coupling would help us relax conservative conditions,  reduce the size of optimization problems on each computational node, and improve communication efficiency with the neighbors. To formally characterize sparsity, we first define the following notations. 

Define $\myset{I}_m \subseteq \myset{I}$ as the index set of the computational nodes involved in the $m$-th constraint, i.e., $i\in \myset{I}_m$ if and only if $\myvar{a}_i^m \neq \myvar{0} \ $ or $ \  b_i^m \neq 0 $. Similarly, define $\myset{M}_i \subseteq \myset{M}$ as the index set of the coupling constraints involving node $i$, i.e., $m\in \myset{M}_i$ if and only if $\myvar{a}_i^m \neq \myvar{0} \ $ or $\  b_i^m \neq 0 $. Define the set of neighboring nodes of node $i$ involved in the $m$-th constraint as $N_i^m = \{j\in \myset{I}:   j\in N_i \ $ and $ \  j\in \myset{I}_m  \}$. In the sparse setting, we assume that each agent knows which neighbors are involved in which relevant constraints, i.e., $N_i^m, \forall m\in \myset{M}_i$ is known to agent $i$.

We call the communication graph $\mathcal{G} = (\myset{I}, E)$ is \textit{consistent} with the $m$-th constraint if the induced subgraph \cite{mesbahi2010graph} $\mathcal{G}_m :=(\myset{I}_m, E_m)$ is connected, where $(i,j)\in E_m $ if and only if  $i,j\in \myset{I}_m$ and $(i,j)\in E $. Define $\mathcal{G}_m^{\prime} :=(\myset{I}, E_m)$.

\begin{assumption} \label{ass:sparsity}
        The communication graph $\mathcal{G}$ is consistent with all constraints.
\end{assumption}
  Figure~\ref{fig:illustrate_sparsity} illustrates a scenario where there are  4 computational nodes and 2 coupling constraints.  Note that both $\mathcal{G}_1 $ (the subgraph in pink) and $\mathcal{G}_2$ (the subgraph in blue) are connected, then $\mathcal{G}$ is consistent with these two constraints. 

\begin{figure}
    \centering
    \includegraphics[width= 0.4\linewidth]{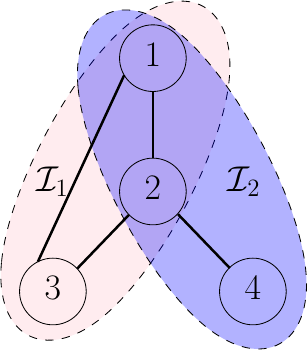}
    \caption{An illustration of a communication graph satisfying Assumption~\ref{ass:sparsity}. In this scenario, the agent index sets for each constraint are $\myset{I}_1 =\{1,2,3\}, \myset{I}_2 =\{1,2,4\} $, the constraint index sets for each agent are $\myset{M}_1 = \myset{M}_2 = \{1,2\}$, $\myset{M}_3 = \{1\}, \myset{M}_4 = \{2\}$, and the sets of neighboring nodes with relevant constraint involvement are $N_1^1 = \{ 2,3\}, N_1^2 = \{ 2\}, N_2^1 = \{ 1,3\}, N_2^2 = \{ 1,4\}, N_3^1 = \{ 1,2\}, N_4^2 = \{ 2\} $, respectively. }
    \label{fig:illustrate_sparsity}
\end{figure}
 Under Assumption~\ref{ass:sparsity},  we know at most $|\myset{M}_i|$ columns of $A_i $ are nonzero. One sufficient condition that fulfills Assumption~\ref{ass:local_slater} in this case is thus
     $\textup{Rank}(A_i) = |\myset{M}_i|$ for all $i$.

\subsection{Equivalent optimization problems}
Now we show that the original problem in \eqref{eq:centralized_problem} is equivalent to the following optimization problems in a sparse setting.
\begin{equation} \label{eq:reformulated_problem_sparsity}
    \begin{aligned}
    &    \hspace{2cm} \min_{\myvar{x}_i, \myvar{y}_i, i\in \myset{I}}  \sum_{i\in \myset{I}} f_i(\myvar{x}_i)  \\
    &  s.t. \ \myvar{a}_i^{m \top} \myvar{x}_i  + \sum_{j\in N_i^m} (y_i^m - y_j^{m}) + b_i^m \leq 0, \\
    & \hspace{4cm} \forall i \in \myset{I}, m\in \myset{M}_i.
    \end{aligned} 
\end{equation}

 \begin{proposition}[Equivalence in a sparse setting]
    Under Assumptions~\ref{ass:local_slater} and ~\ref{ass:sparsity}, the two problems in \eqref{eq:centralized_problem} and \eqref{eq:reformulated_problem_sparsity} are equivalent, in the sense that
    \begin{enumerate}
        \item for any feasible solution   $(\myvar{x}^{\prime}, \myvar{y}^{\prime}  )$ to \eqref{eq:reformulated_problem_sparsity}, $\myvar{x}^{\prime}$ is a feasible solution to \eqref{eq:centralized_problem};
        \item for any feasible solution $\myvar{x}^{\prime}$ to \eqref{eq:centralized_problem}, there exists a variable  $\myvar{y}^{\prime} $ such that $(\myvar{x}^{\prime}, \myvar{y}^{\prime}  )$ is a feasible solution to \eqref{eq:reformulated_problem_sparsity};
        \item the two problems have the same cost function.
    \end{enumerate}
\end{proposition}
The proof follows similar steps as in that of Proposition~\ref{prop:equivalent_problem} except that now we need to use the Laplacian matrices of the induced subgraphs $\mathcal{G}_m, m\in \myset{M}$, which are assumed to be connected under Assumption~\ref{ass:sparsity}. Details are omitted here for brevity. Similar to Section~\ref{sec:equiv_program}, the reformulated problem in \eqref{eq:reformulated_problem_sparsity} has a separable structure, and  only local information $y_j^m, j\in N_i^m$ is needed to determine the constraints concerning each node.

 \begin{figure*}[!t]
    \normalsize
    \setcounter{MYtempeqncnt}{\value{equation}}
    \setcounter{equation}{\value{MYtempeqncnt}}
    \begin{equation} \label{eq:sparse_sensitivity}
    \begin{aligned}
    \phi_i(\myvar{y}^\prime) &\geq \phi_i(\myvar{y}) + {\textstyle \sum}_{m\in \myset{M}_i} c_i^m {\textstyle \sum}_{j\in N_i^m} \left((y_i^{\prime,m} - y_j^{\prime,m}) - ( y_i^m  - y_j^{m})\right)  \\
        & = \phi_i(\myvar{y}) +{\textstyle \sum}_{m\in \myset{M}_i} c_i^m \tilde{\myvar{l}}_{i,\mathcal{G}_m}   ( \myvar{y}^{\prime, m} - \myvar{y}^m)  \\
        & \hspace{4.5cm}+ {\textstyle \sum}_{m\notin \myset{M}_i} c_i^m \tilde{\myvar{l}}_{i,\mathcal{G}_m}   ( \myvar{y}^{\prime, m} - \myvar{y}^m) \\
        & = \phi_i(\myvar{y}) + \myvar{c}_i^{\top} \tilde{L}_i    ( \myvar{y}^{\prime} - \myvar{y})
    \end{aligned}, \text{ where }  \tilde{L}_i  := \begin{bmatrix}
    \tilde{\myvar{l}}_{i,\mathcal{G}_1} & 0 & ... & 0 \\
    0 & \tilde{\myvar{l}}_{i,\mathcal{G}_2} & ... & 0 \\
    0 & 0 & ... & 0 \\
    0 & 0 & ... & \tilde{\myvar{l}}_{i,\mathcal{G}_M} 
    \end{bmatrix}
    \end{equation}
   
    \setcounter{equation}{\value{MYtempeqncnt}+1}
    \hrulefill
    \vspace*{4pt}
    \end{figure*}

It is obvious from \eqref{eq:reformulated_problem_sparsity} that  $y_i^m, m\in \myset{M} \setminus \myset{M}_i$ is not involved in the local optimization problem. Thus,  we can pre-set those components in $\myvar{y}$ to be constant zero. The sparse formulation in \eqref{eq:reformulated_problem_sparsity} also accounts for local constraints that involve only one node. Denoting the node by $i$, we have $a_j^m, b_j^m = 0$ for any $j\neq i$ and we pre-set $y_k^m = 0$ for any $k\in \mathcal{I}$.

\subsection{local optimization problems}
Following a similar analysis procedure, in the following we will treat $\myvar{y}_i$ as an independent variable and design a distributed updating law so that $\myvar{y}_i$ converges to the optimal $\myvar{y}_i^\star$ to \eqref{eq:reformulated_problem_sparsity}. For node $i$ and a given $\myvar{y} $, 
 a local optimization problem is given by
 \begin{equation} \label{eq:sparse_local_QP_phi}
   \begin{aligned}
    & \hspace{3cm} \phi_i(\myvar{y}) :=   \min_{\myvar{x}_i}  f_i(\myvar{x}_i)  \\
    &  s.t. \ \myvar{a}_i^{m \top} \myvar{x}_i  + \sum_{j\in N_i^m} (y_i^m - y_j^{m}) + b_i^m \leq 0,  \forall m\in \myset{M}_i.
    \end{aligned} 
\end{equation}
Define $\myvar{c}_i = (c_i^1, ..., c_i^m, ..., c_i^M)$, where $c_i^m$ is the Lagrange multiplier of \eqref{eq:sparse_local_QP_phi} if $m\in \myset{M}_i$, and is equal to $0$ otherwise. 

Similar to previous discussions, we have the constraint satisfaction property as follows. The proof is omitted for brevity. 
\begin{proposition}[Constraint satisfaction]
    For any given $\myvar{y}$, if $\myvar{x}_i, i\in \myset{I}$ is feasible to the local optimization problem in \eqref{eq:sparse_local_QP_phi}, then  $[\myvar{x}_1; \myvar{x}_2; ...; \myvar{x}_N]$ is a feasible solution to \eqref{eq:centralized_problem}.
\end{proposition}

Define the summation of the optimal values of the local optimization problems as
\begin{equation} \label{eq:sparse_phi_y}
    \phi(\myvar{y}) = \sum_{i\in \myset{I}} \phi_i(\myvar{y}).
\end{equation}

\begin{proposition}[Subdifferential] \label{prop:subdifferential_sparse}
    Let Assumptions~\ref{ass:convex_f},~\ref{ass:finite_cost} ,~\ref{ass:local_slater} and~\ref{ass:sparsity} hold.  Let $\myvar{c}_i= (c_i^1, ..., c_i^M)\in \mathbb{R}^m$ be defined below \eqref{eq:sparse_local_QP_phi}. Define $\zeta_i^m = \sum_{j\in N_i^m} (c_i^m - c_j^m)$ if $m\in \myset{M}_i$ and $\zeta_i^m = 0$ if $m\notin \myset{M}_i$. Let $\myvar{\zeta}^m = (\zeta_1^m, \zeta_2^m,..., \zeta_N^m), \myvar{\zeta} = \begin{bmatrix}
    \myvar{\zeta}^1;  \myvar{\zeta}^2;  ...;  \myvar{\zeta}^M
\end{bmatrix}$. Then 
\begin{enumerate}
    \item $\myvar{c}_i\in \mathbb{R}^m$, not necessarily unique, always exists,
    \item $\phi(\myvar{y})$ is convex,
    \item $\myvar{\zeta} \in \partial \phi(\myvar{y}).$
\end{enumerate}
   
\end{proposition}

\begin{proof}

    The first two claims follow from a similar argument of that of Proposition~\ref{prop:subdifferential} and are omitted here. We will show Point 3) in the following. From sensitivity analysis, for any $\myvar{y}, \myvar{y}^\prime$, we know \eqref{eq:sparse_sensitivity} holds, where $\tilde{\myvar{l}}_{i,\mathcal{G}_m} $ is the $i$-th row of the Laplacian matrix of $\mathcal{G}_m^\prime = (\myset{I}, E_{m})$.  The first inequality in \eqref{eq:sparse_sensitivity} is from \cite[Equation 5.57]{boyd2004convex}, and the last two equalities follow from the definition of graph Laplacian matrices and noting that when $m\notin \myset{M}_i$, $\tilde{\myvar{l}}_{i,\mathcal{G}_m}   = \myvar{0}$.  \eqref{eq:sparse_sensitivity} implies that $\tilde{L}_i^{\top}\myvar{c}_i$ is a subgradient of $\phi_i(\myvar{y})$.

    Following the calculus rule for subdifferential,  $\sum_i \tilde{L}_i^{\top} \myvar{c}_i$ is a subgradient of $\phi(\myvar{y})$. With a few linear algebraic operations, we know 
    \begin{equation*}
        \tilde{L}_i^{\top} \myvar{c}_i  = \begin{bmatrix}
    \tilde{\myvar{l}}_{i,\mathcal{G}_1}^{\top} & 0 & ... & 0 \\
    0 & \tilde{\myvar{l}}_{i,\mathcal{G}_2}^{\top} & ... & 0 \\
    0 & 0 & ... & 0 \\
    0 & 0 & ... & \tilde{\myvar{l}}_{i,\mathcal{G}_M}^{\top} 
    \end{bmatrix} \begin{bmatrix}
    c_i^1\\
    c_i^2 \\
    ...\\
    c_i^M
\end{bmatrix} =  \begin{bmatrix}
            c_i^1 \tilde{\myvar{l}}_{i,\mathcal{G}_1}^{\top} \\
             c_i^2 \tilde{\myvar{l}}_{i,\mathcal{G}_2}^{\top} \\
            ... \\
           c_i^M \tilde{\myvar{l}}_{i,\mathcal{G}_M}^{\top} 
        \end{bmatrix}
    \end{equation*}
    Thus, the entry in $\sum_i \tilde{L}_i^{\top} \myvar{c}_i$ that corresponds to $y_i^m$ is 
    \begin{multline} \label{eq:sparse_subdifferentiable}
        \left[\displaystyle \sum_{k\in \myset{I}}  c_k^m \tilde{\myvar{l}}_{k,\mathcal{G}_m}^{\top} \right]_i = \sum_{k\in \myset{I}}\left[\tilde{\myvar{l}}_{k,\mathcal{G}_m}^{\top}\right]_i c_k^m \\=
    \begin{cases}
        \sum_{j\in N_i^m}( c_i^m - c_j^m) & \text{if } m\in \myset{M}_i\\
         0 & \text{if } m\notin \myset{M}_i
    \end{cases}.
    \end{multline}
    The second equality comes from the fact that $[\myvar{l}_{k,\mathcal{G}_m}^{\prime,\top}]_i = [L_{\mathcal{G}_m^\prime}]_{k,i}$, the $(k,i)$-th element of the Laplacian of the graph $\mathcal{G}_m^\prime = (\myset{I}, E_{m})$. Thus $\left[\myvar{l}_{k,\mathcal{G}_m}^{\prime,\top}\right]_i$ is equal to $0$ if $m\notin \myset{M}_i$, or $k\neq i \  \wedge \ k\notin N_i^m$;   $-1$ if $k\in N_i^m$; and  $|N_i^m|$ if $k= i$.  Thus we  conclude $\myvar{\zeta} $ is a subgradient of $\phi(\myvar{y})$.    
\end{proof}

The distributed optimization algorithm  for node $i$ in the sparse case is given as follows.

\begin{algorithm} 
\caption{Sparse Distributed Optimization Algorithm}
\label{alg:sparse_distributed_optimization}
\begin{algorithmic}[1]
\State \textbf{Local stored state:}  $\myvar{x}_i, \myvar{y}_i$
\State  \textbf{Initialization:} Choose  arbitrary $\myvar{y}_i$, and let $y_i^m = 0$ if $m\notin \myset{M}_i$.
\Loop
\State Send $y_i^m$ to and gather $y_j^m$ from $j\in N_i^m, \forall m \in \mathcal{M}_i$
\State Compute $(\myvar{x}_i, \myvar{c}_i)$ from \eqref{eq:sparse_local_QP_phi} 
\State Send $c_i^m$ to and gather $c_j^m$ from $j\in N_i^m, \forall m \in \mathcal{M}_i$
\State Update $\myvar{y}_i$ according to \eqref{eq:spares_adaptive law}
\EndLoop
\end{algorithmic}
\end{algorithm}

The main difference compared to Algorithm~\ref{alg:distributed_optimization} is that 1) now node $i$  receives $\myvar{y}_j, \myvar{c}_j$ from and send $\myvar{y}_i, \myvar{c}_i$ to neighboring nodes that are relevant, as in set $N_i^m$ in Steps 4 and 7; and 2) only relevant components in $\myvar{y}_j$ and $\myvar{c}_j$ are taken into account, as in \eqref{eq:sparse_local_QP_phi} and \eqref{eq:spares_adaptive law} in Steps 5 and 6. This leads to fewer variables to communicate among the network and smaller size local optimization problems to solve.

The following theorem states the convergence property.
\begin{thm}
    Let Assumptions~\ref{ass:convex_f},~\ref{ass:finite_cost},~\ref{ass:local_slater},~\ref{ass:sparsity} hold and $\myvar{c}_i$ as defined below \eqref{eq:sparse_local_QP_phi}. If $\myvar{y}_i(t)$  is updated according to 
\begin{equation} \label{eq:spares_adaptive law}
   \dot{y}_i^m = \begin{cases}
       - k_0 \sum_{j\in N_i^m } (c_i^m - c_j^m),  & \text{if } m\in \myset{M}_i \\
       0 & \text{if } m\notin \myset{M}_i
   \end{cases},
\end{equation}
for almost all $t$, where $k_0>0$ is a constant gain and a Caratheodory solution exists, then the properties in Theorem~\ref{thm:convergence and constraints} hold.
\end{thm}

\begin{proof}

   In view of \eqref{eq:sparse_subdifferentiable}, we observe that $\dot{y}_i^m(t)$ given in \eqref{eq:spares_adaptive law} is exactly the product of \eqref{eq:sparse_subdifferentiable} together with a constant negative gain. Stacking $\dot{y}_i^m$ by the order of constraints, we obtain  that $ \dot{\myvar{y}}\in  -\partial k_0 \phi(\myvar{y})$ for almost all $t$ from Proposition~\ref{prop:subdifferential_sparse}.
     Thus, the analyses in the proof of Theorem~\ref{thm:convergence and constraints} are also applicable here and are neglected.
\end{proof}

\section{A special case: control barrier functions-induced quadratic programs} \label{sec:qp_case}

 \subsection{Control barrier functions and safe distributed control}
 In this section, we consider the application of our proposed algorithms in the context of control barrier function-induced safe distributed control. CBF approaches \cite{Wang2017a,fernandez2023distributed, mestres2024distributed} have already been applied to multi-robot safe coordination with pair-wise constraints, for example, collision avoidance and connectivity maintenance. In the following we introduce a more general formulation that further includes  collective constraints.
 
 Consider a multi-agent system with $\myvar{p}_i$ the state vector of agent $i, i\in \myset{I}$. The dynamics of each agent is given by $$\dot{\myvar{p}}_i = \mathfrak{f}_i(\myvar{p}_i) + \mathfrak{g}_i(\myvar{p}_i) \myvar{x}_{i},$$ where $\mathfrak{f}_i, \mathfrak{g}_i $ are locally Lipschitz vector fields, $\myvar{x}_{i}$ is the control input to agent $i$ that is yet to be designed. Suppose that this multi-agent system is subject to $M$ coupling state constraints, that is, along the system trajectory, it has to maintain $$h^m(\myvar{p}_1,...,\myvar{p}_N)\geq 0, m\in \myset{M},$$ and $h^m$ is a smooth function. The constraint $h^m$ is a collective constraint since it restricts the states of all agents. We call  the following inequality a CBF condition associated with the $m$-th constraint
    \begin{equation} \label{eq:cbf_cond}
        \sum_{i\in \myset{I}}\myvar{a}_i^{m\top} \myvar{x}_{i} + b_i^m\leq  0,
  \end{equation}
  where $\myvar{a}_i^{m\top} = -\nabla_{\myvar{p}_i} h^{m\top} \mathfrak{g}(\myvar{p}_i)$, $b_i^m, i\in \myset{I}$ are such that $ \sum_{i\in \myset{I} }  b_i^m = -k h^m (\myvar{p}_1,...,\myvar{p}_N) - \sum_{i\in \myset{I}} \nabla_{\myvar{p}_i} h^{m\top}  \mathfrak{f}_i(\myvar{p}_i)$, and constant $k>0$. Here we use linear functions instead of the extended class $\mathcal{K}$ functions for simplicity. Results from  CBF literature \cite{Ames2019control} state that, if a locally Lipschitz control input $\myvar{x}_i, i\in \myset{I},$ is chosen such that \eqref{eq:cbf_cond}
  holds for all time, then the set $\{(\myvar{p}_1, \myvar{p}_2, ..., \myvar{p}_N): h^m(\myvar{p}_1,...,\myvar{p}_N)\geq 0\}$ is forward invariant and the $m$-th state constraint is always satisfied, provided that it is initially satisfied.  In the following analysis we will neglect the explicit state constraints $h^m(\myvar{p}_1, \myvar{p}_2, ..., \myvar{p}_N)\geq 0$ and instead focus on enforcing the CBF conditions for all time.

 In CBF literature \cite{Ames2019control,tan2021distributed}, a commonly used formulation to enforce the CBF conditions is expressed as a quadratic program as follows
 \begin{equation} \label{eq:CBF_problem}
\begin{split}
    &\min_{\myvar{x}_1,\dots,\myvar{x}_N} \sum_{i\in\myset{I}}  \frac{1}{2}\| \myvar{x}_i - \myvar{x}_{nom,i}\|^2 \\
    & {s.t.} \quad \left\{
    \begin{array}{cc}
      \sum_{i\in\myset{I}} {\myvar{a}_i^{1 \top}}\myvar{x}_i + b_i^1  \leq 0, \\
         \vdots \\
       \sum_{i\in\myset{I}} {\myvar{a}_i^{M \top}}\myvar{x}_i + b_i^M  \leq 0,
    \end{array}
    \right.
\end{split}
\end{equation}
   The quadratic program in \eqref{eq:CBF_problem} aims at minimally modifying nominal local controllers $\myvar{x}_{nom,i}$, which relate to the underlying system tasks, while always respecting $M$ CBF conditions, hereafter referred to as coupling constraints, which relate to system safety.

 One challenging problem that naturally arises from this formulation is, for multi-agent systems where agents only have local information, how to solve this optimization problem in a distributed way with guaranteed safety assurance.

  We assume that $\myvar{a}_i, b_i, \myvar{x}_{nom,i}$ only depend on locally obtainable information from the  neighbors of the $i$-th agent, and that the quadratic program is always feasible. These are referred to as the local obtainability and compatibility properties of the CBF conditions as discussed in \cite{tan2021distributed, tan2022compatibility}. One illustrative application of coordinating a multi-agent system with multiple state constraints is shown later in the simulation section.

  What sets apart the CBF-induced QP in \eqref{eq:CBF_problem} from the previously discussed distributed optimization formulation is that, since the solution to \eqref{eq:CBF_problem} is applied as feedback to a dynamical system and $\myvar{x}_{nom,i}, \myvar{a}_i^m, b_i^m$ are state-dependent, they are time-varying in nature rather than static and evolve along the closed-loop system trajectory. In the following, we will first take a look at a time-invariant instance of \eqref{eq:CBF_problem}, i.e., when $\myvar{x}_{nom,i}, \myvar{a}_i^m$ and $ b_i^m$ are constant, and then analyze the performance of our proposed algorithm in a time-varying setting. We note that the results in this section are also applicable to the sparsely coupled case. This extension is straightforward and neglected for notation simplicity.

  Following the previous analysis, with the help of an auxiliary variable $\myvar{y}\in \mathbb{R}^{NM}$, we have $\phi(\myvar{y}) = \sum_{i\in \myset{I}} \phi_i(\myvar{y})$ with $ \phi_i(\myvar{y})$ given by the local optimization problems
 \begin{equation} \label{eq:local_problem_qp}
\begin{split}
    \phi_i(\myvar{y}) = &\min_{\myvar{x}_i} \frac{1}{2} \lVert \myvar{x}_i - \myvar{x}_{nom,i} \rVert^2 \\
    & {s.t. }  \ {A}_i^\top \myvar{x}_i + ({I}_M \otimes \myvar{l}_i)\myvar{y} + \myvar{b}_i\leq 0.
\end{split}
\end{equation}

\begin{proposition} \label{prop:local analytical solution}
Consider the  local QP in \eqref{eq:local_problem_qp} with constant $\myvar{x}_{nom,i}, A_i, \myvar{y}, \myvar{b}_i$. Assume that Assumption~\ref{ass:local_slater} holds. Then the Lagrange multiplier $\myvar{c}_i = (c^1_i, c^2_i, ..., c^M_i)$ exists and satisfies the algebraic equation
 \begin{multline}\label{eq:local_C}
        \Lambda_i \myvar{c}_i = \max\left( {A}_i^\top \myvar{x}_{nom,i}+ ({I}_M \otimes \myvar{l}_i)\myvar{y}\right.\\  +\myvar{b}_i \left. + ( \Lambda_i - {A}_i^\top {A}_i)\myvar{c}_i ,\myvar{0}\right),
 \end{multline} 
 where $\Lambda_i = \textup{diag}({A}_i^\top {A}_i)$, and the primal-dual optimizer to \eqref{eq:local_problem_qp} is $( \myvar{x}_{nom,i} - A_i \myvar{c}_i, \myvar{c}_i )$. The optimal cost is $\frac{1}{2}\| A_i \myvar{c}_i\|^2$.
 
\end{proposition}
The proof follows from analyzing the Karush–Kuhn–Tucker (KKT) condition of the Lagrangian and is given in Appendix for the sake of completeness. A similar result was reported in \cite[Theorem 5]{santilli2020distributed}. There however  $A_i, \Lambda_i $ are assumed to be non-singular.

For the time-invariant optimization problem, if the local optimization problem in \eqref{eq:local_problem_qp} always fulfills the Linearly Independent Constraint Qualification\footnote{We say that the Linearly Independent Constraint Qualification is satisfied at $\myvar{x}^\star$ for an optimization problem, if the gradients of all active constraints are linearly independent.} (LICQ) by, for example, $A_i$ being full column rank, then \cite[Corollary 7.3.1]{florenzano2001finite} 
states that $\phi_i(\myvar{y})$ in \eqref{eq:local_problem_qp} is differentiable as discussed in Remark~\ref{rem:differentiability}. In this case, the subgradient reduces to the gradient in the normal sense, and the adaptive law in \eqref{eq:adaptive law}  becomes a gradient flow. Under Assumptions~\ref{ass:finite_cost},~\ref{ass:local_slater} and the convexity of $\phi(\myvar{y})$, the convergence guarantee of $\phi(\myvar{y}(t))$ is well-established in view that it is lower bounded and monotonically decreasing.

\subsection{Finite-time convergence for time-varying optimizations}
In this subsection, we will strengthen our proposed algorithms of the previous sections to compensate for the time-varying parameters in the QP. Under LICQ, since the update law in \eqref{eq:adaptive law} is a gradient flow, one could  modify it  to a scaled normalized gradient flow as in \cite[(6b)]{cortes2006finite}, that is, 
\begin{equation} \label{eq:adaptive law finite time}
    \dot{\myvar{y}}_i = - k_0 \sign{\frac{\partial}{\partial \myvar{y}_i} }\phi(\myvar{y}) =  - k_0 \sign \sum_{j\in N_i} (\myvar{c}_i - \myvar{c}_j). 
\end{equation}
\cite{cortes2006finite} shows that the normalized gradient flow achieves finite-time convergence for static strongly convex functions. However, this result is not directly applicable here since $\phi(\myvar{y})$ is not strongly convex. Moreover, for the CBF-induced QP in \eqref{eq:CBF_problem},  $\myvar{x}_{nom,i}, \myvar{a}_i^m, b_i^m$ depend on the states of the multi-agent system and are time-varying. Before presenting the main result, we introduce the following Lemma that will simplify the analysis later on.
\begin{lemma} \label{lem:rearrangement}
    Let  $\myvar{w}_C$ be the vector of $\{ w_i^m \}_{i\in \myset{I}, m\in \myset{M}}, w_i^m\in \mathbb{R} $ 
 ordered by constraints. Define $\myvar{v}_C := I_M\otimes L\myvar{w}_C,  \myvar{v}_A := [I_M\otimes \myvar{l}_1;I_M\otimes \myvar{l}_2;...;I_M\otimes \myvar{l}_N] \myvar{w}_C $, where $\myvar{l}_i$ is the $i$-th row of the Laplacian $L$. Then, $\myvar{v}_C = J \myvar{v}_A$, where $J$ is the permutation matrix defined in Section~\ref{sec:setup}.A. 
\end{lemma}
\begin{proof}
    Define $\myvar{w}^m = (w_1^m, ..., w_N^m)$. Let $v_i^m = [L\myvar{w}^m]_i$. Then it is clear that $\myvar{v}_C$ is the vector of $\{v_i^m \}_{i\in \myset{I}, m\in \myset{M}}$ ordered by constraints, and $\myvar{v}_A$ is the vector of $\{v_i^m \}_{i\in \myset{I}, m\in \myset{M}}$ ordered by agents. Thus $\myvar{v}_C = J \myvar{v}_A$.
\end{proof}
 The main result is summarized in the theorem below with slowly time-varying $\myvar{x}_{nom,i}(t), \myvar{a}_i^m(t), b_i^m(t)$. Here we note that LICQ is imposed as an easy-to-check sufficient condition for guaranteeing the differentiability of $\phi_i(\myvar{y})$. Depending on the problem at hand, other conditions may also be applicable.

\begin{thm} \label{thm:time-varying_cbf_qp_properties}
  Consider the centralized QP in \eqref{eq:CBF_problem}. Assume the LICQ holds for every local optimization problem in \eqref{eq:local_problem_qp} at every time instant. Let $\myvar{c}_i$ be the Lagrange multiplier of the local QP for agent $i$.
  Assume further that $\myvar{x}_{nom,i}(t), \myvar{a}_i^m(t), b_i^m(t)$ are slowly time-varying in the sense that there exists a constant $D>0$ such that $\| \myvar{d}_i^{\mathbf{I}_i}(t) \| \leq D $ for all $ i\in \myset{I} $ and all index sets $\mathbf{I}_i\subseteq \myset{M}$ with 
  \begin{equation} \label{eq:d_i_I}
      \myvar{d}_i^{\mathbf{I}_i}(t) = -\frac{d A_i^{\top,\mathbf{I}_i} A_i^{\top, \mathbf{I}_i,\top}}{dt} \myvar{c}_i^{\mathbf{I}_i} + \frac{d}{dt}(A_i^{\top,\mathbf{I}_i} \myvar{x}_{nom,i} + \myvar{b}_i^{\mathbf{I}_i} ).
  \end{equation}
 Denote $\underline{\lambda} = \min_{\mathbf{I}_i,i\in \myset{I}}\lambda_{\min}( ({A}_i^{\top,\mathbf{I}_i} {A}_i^{\top,\mathbf{I}_i,\top})^{-1})$, $\bar{\lambda} = \max_{\mathbf{I}_i,i\in \myset{I}}\lambda_{\max}( ({A}_i^{\top,\mathbf{I}_i} {A}_i^{\top,\mathbf{I}_i,\top})^{-1})$. If $k_0$ is chosen such that 
  \begin{equation} \label{eq:k_0}
      k_0 \underline{\lambda} - ND\sqrt{M}\bar{\lambda} \| L\| - \epsilon \geq 0
  \end{equation}
  for some $\epsilon>0$,  then the Filippov solution of the scaled normalized gradient flow in \eqref{eq:adaptive law finite time} converges to a critical point of $\phi(\myvar{y})$ in finite time $t_r\leq  \sum_{m\in \myset{M}}\| L\myvar{c}^m(0)\|_1/\epsilon$. Moreover, the coupling constraints in \eqref{eq:CBF_problem} are always satisfied.
\end{thm}

\begin{proof}
 Consider the nonsmooth Lyapunov function $$V(\myvar{y}) = \left\lVert \frac{\partial \phi(t,\myvar{y})}{\partial \myvar{y}} \right\lVert_1 = \| (I_M\otimes L) \myvar{c}_C \|_1 = \sum_{m\in \myset{M}}\| L\myvar{c}^m\|_1.$$  Denote at the time $t$ the index set of strictly active constraints of agent $i$ to be $\mathbf{I}_i(t)$ and its complement $\bar{\mathbf{I}}_i(t) = \myset{M}\setminus \mathbf{I}_i(t)$. 

 In the following, we will show that for any constant $\mathbf{I}_i(t), i\in \myset{I}$, $V(\myvar{y})$  decreases monotonically for a system $\Sigma_{\mathbf{I}_i, i\in \myset{I}}$ given by \eqref{eq:adaptive law finite time} together with the algebraic equation 
 \begin{equation} \label{eq:c_I given y}
\begin{aligned}
         {A}_i^{\top,\mathbf{I}_i} {A}_i^{\top,\mathbf{I}_i,\top}\myvar{c}_i^{\mathbf{I}_i} & = {A}_i^{\top,\mathbf{I}_i} \myvar{x}_{nom,i}+ ({I}^{\mathbf{I}_i} \otimes \myvar{l}_i)\myvar{y}+\myvar{b}_i^{\mathbf{I}_i}, \\
         \myvar{c}_i^{\bar{\mathbf{I}}_i} &= \myvar{0}.
\end{aligned}
\end{equation}
for all $ i\in \myset{I}$. This equation is obtained by considering only the constraints indexed in $\mathbf{I}_i$ in \eqref{eq:local_C} and noting $\myvar{c}_i^{\mathbf{I}_i}> \myvar{0}$. When viewing $\mathbf{I}_i, i\in \myset{I}$ as the index of such systems, since $V(\myvar{y})$  decreases monotonically for all $\Sigma_{\mathbf{I}_i, i\in \myset{I}}$ as will be shown later,  $V(\myvar{y})$ effectively acts as a common Lyapunov function \cite{liberzon2003switching}. 
 Thus it suffices to consider a system $\Sigma_{\mathbf{I}_i, i\in \myset{I}}$, and the time index is dropped to ease the notation.

For arbitrary constant $\mathbf{I}_i, i\in \myset{I}$, from \eqref{eq:c_I given y}, we know 
     \begin{equation*} 
    \begin{aligned}
       &  d \myvar{c}_i^{\mathbf{I}_i}/d t = ({A}_i^{\top,\mathbf{I}_i} {A}_i^{\top,\mathbf{I}_i,\top})^{-1}({I}^{\mathbf{I}_i} \otimes \myvar{l}_i d \myvar{y}/dt  + d_i^{\mathbf{I}_i}) \\
        & d \myvar{c}_i^{\bar{\mathbf{I}}_i}/d t = \myvar{0}
    \end{aligned}
\end{equation*}
Here $d_i^{\mathbf{I}_i}$ is defined in \eqref{eq:d_i_I}. ${A}_i^{\top,\mathbf{I}_i} {A}_i^{\top,\mathbf{I}_i,\top}$ is invertible since LICQ is assumed. Denote by $\Gamma_i^\prime =                  \begin{pmatrix}
            ({A}_i^{\top,\mathbf{I}_i} {A}_i^{\top,\mathbf{I}_i,\top})^{-1} & 0 \\
            0 & 0
        \end{pmatrix} $ and  $J_{\mathbf{I}_i}$ the permutation matrix such that $[\myvar{c}_i^{\mathbf{I}_i}; \myvar{c}_i^{\bar{\mathbf{I}}_i} ] = J_{\mathbf{I}_i} \myvar{c}_i $. Thus $[I^{\mathbf{I}_i}; I^{\bar{\mathbf{I}}_i}] = J_{\mathbf{I}_i}  I_M$. We can write the time derivative of $\myvar{c}_i$ in a compact form as
        \begin{equation} \label{eq:dc_I_i / dy time-varying}
        \begin{aligned}
               \dot{\myvar{c}}_i & = J_{\mathbf{I}_i }^\top [d \myvar{c}_i^{\mathbf{I}_i}/d t; d \myvar{c}_i^{\bar{\mathbf{I}}_i}/d t] \\
               & = J_{\mathbf{I}_i }^\top (\Gamma_i^\prime)J_{\mathbf{I}_i} ( I_M\otimes\myvar{l}_i \dot{\myvar{y}} + J_{\mathbf{I}_i}^\top [\myvar{d}_i^{\mathbf{I}_i}; \myvar{0}]) \\
               & = \Gamma_i(I_M\otimes\myvar{l}_i \dot{\myvar{y}} + \myvar{d}_i^{\mathbf{I}_i,\prime })
        \end{aligned}
        \end{equation}
where $\Gamma_i = J_{\mathbf{I}_i}^\top \Gamma_i^\prime J_{\mathbf{I}_i}, \myvar{d}_i^{\mathbf{I}_i,\prime }:= J_{\mathbf{I}_i}^\top [\myvar{d}_i^{\mathbf{I}_i}; \myvar{0}]$.

Let $\myvar{z} = I_M\otimes L \myvar{c}_C$.  Then $V(\myvar{y}) = \| \myvar{z}\|_1$. The generalized time derivative is given by
\begin{equation} \label{eq:nonsmooth_V_dot}
\begin{aligned}
        \frac{d}{dt}{V}(\myvar{y}) & \in \sum_{i\in \myset{I}^{\neq }} \sign{(z_i)} \dot{z}_i + \sum_{i\in \myset{I}^{=}} \text{SIGN}{(z_i)} \dot{z}_i 
\end{aligned}
\end{equation}
 where $\myset{I}^= = \{i\in \myset{I}:z_i = 0 \}$, $\myset{I}^{\neq} = \{i\in \myset{I}:z_i \neq 0 \}$, $\textup{SIGN}(z)  :=\left\{ \begin{smallmatrix} 1, &z >0; \\
[-1,1], &z = 0;\\
-1, &z<0.\end{smallmatrix}\right.$
Since we are dealing with Filippov solutions, we can disregard the case which $z_i = 0$ holds for isolated time instants of measure zero. If $z_i = 0$ holds along an interval of time of positive measure, then, in the sense of Filippov, $\dot{z}_i$ exists at those time instants and $\dot{z}_i = 0$. Based on this fact,  we have that for almost all $t$, $    \frac{d}{dt} V(\myvar{z})= \sum_{i\in \myset{I}^{\neq}} \sign(z_i)\dot{z}_i + \sum_{i\in \myset{I}^{=}} \textup{sign}(z_i)\dot{z}_i$, that is,

\begin{equation} \label{eq:nonsmooth_V_dot_time-varying}
\begin{aligned}
        \frac{d}{dt}{V}(\myvar{y}) 
        & = \sign{\myvar{z}}^\top  \dot{\myvar{z}} = \sign{\myvar{z}}^\top (I_M\otimes L) J \dot{\myvar{c}}_A  
\end{aligned}
\end{equation}

Define $\myvar{\sigma} =  (I_M\otimes L) \sign{\myvar{z}}$. After some calculations, one finds $\myvar{\sigma} = [ L\sign{L \myvar{c}^1};  L\sign{L \myvar{c}^2}; ...;  L\sign{L \myvar{c}^M} ]$. Denote by $\sigma_{i}^m = [L\sign{L \myvar{c}^m}]_i $. Then $\myvar{\sigma}$ can be seen as $ \{\sigma_{i}^m\}_{i\in \myset{I}, m\in \myset{M}}$ ordered by constraints. Let $\myvar{\sigma}_i = (\sigma_{i}^1, \sigma_{i}^2,..., \sigma_{i}^M)$, and $\myvar{\sigma}_A = J^\top \myvar{\sigma} = [\myvar{\sigma}_1; \myvar{\sigma}_2; ...; \myvar{\sigma}_N]$ be the vector of $ \{\sigma_{i}^m\}_{i\in \myset{I}, m\in \myset{M}}$ ordered by agents. From Lemma~\ref{lem:rearrangement}, we have 
\begin{equation} \label{eq:sigma_A}    
\myvar{\sigma}_A = [I_M\otimes \myvar{l}_1;I_M\otimes \myvar{l}_2;...;I_M\otimes \myvar{l}_N] \sign{\myvar{z}}.
\end{equation}
Furthermore, from \eqref{eq:adaptive law finite time}, $\dot{\myvar{y}}_A = -k_0 \sign{}([I_M\otimes \myvar{l}_1;I_M\otimes \myvar{l}_2;...;I_M\otimes \myvar{l}_N] \myvar{c}_C )$, thus $\dot{\myvar{y}} = -k_0 \sign{(I_M\otimes L\myvar{c}_C)} = - k_0 \sign{(\myvar{z})}$ in view of Lemma~\ref{lem:rearrangement} and the definition of $\myvar{z}$. Further considering \eqref{eq:sigma_A}, we have
$[I_M\otimes \myvar{l}_1;I_M\otimes \myvar{l}_2;...;I_M\otimes \myvar{l}_N] \dot{\myvar{y}} = -k_0 \myvar{\sigma}_A. $ Following \eqref{eq:dc_I_i / dy time-varying}, we thus obtain
\begin{multline}
        \dot{\myvar{c}}_A = [\dot{\myvar{c}}_1; \dot{\myvar{c}}_2; ...; \dot{\myvar{c}}_N] \\ =\text{blkdiag}(\Gamma_1, ..., \Gamma_N) (-k_0 \myvar{\sigma}_A + [\myvar{d}_1^{\mathbf{I}_1,\prime};...; \myvar{d}_N^{\mathbf{I}_N,\prime}]) 
\end{multline}

From \eqref{eq:nonsmooth_V_dot_time-varying}, the generalized time derivative is
\begin{equation} \label{eq:nonsmooth_V_dot_time-varying2}
\begin{aligned}
        \frac{d}{dt}{V}(\myvar{y}) 
        & =   \myvar{\sigma}_A^\top \dot{\myvar{c}}_A \\
        & = \sum_{i\in \myset{I}} \myvar{\sigma}_i^\top \Gamma_i (-k_0 \myvar{\sigma}_i  + \myvar{d}_i^{\mathbf{I}_i,\prime}) \\
\end{aligned}
\end{equation}

 Consider the following two cases: 
\begin{enumerate}
    \item $L\myvar{c}^m= \myvar{0}$ for all $m\in \myset{M}$. In this case, $V(\myvar{y}) =0$, and thus from the first-order optimality condition one critical point of $\phi(\myvar{y})$ has been reached.
    \item $L\myvar{c}^m\neq \myvar{0}$ for some $m \in \myset{M}$. This implies that $\| L\sign{L\myvar{c}^m}\| \geq 1$ \cite[Proposition 2.1]{franceschelli2014finite}. Without loss of generality, assume that $| \sigma_1^1 |\geq 1$  and   $\mathbf{I}_1 = \{1,2,.., M'\}$.  These assumptions are not restrictive as they can be fulfilled by permuting the numbering of the agents and constraints. This leads to
    \begin{equation}
        \Gamma_1 = \begin{pmatrix}
            ({A}_1^{\top,\mathbf{I}_1} {A}_1^{\top,\mathbf{I}_1,\top})^{-1} & 0 \\
            0 & 0
        \end{pmatrix} 
    \end{equation}
    and $\| \myvar{\sigma}_1^{\mathbf{I}_i} \|\geq 1$. 
    The generalized time derivative becomes
    \begin{equation} \label{eq:nonsmooth_V_dot_time-varying3}
    \begin{aligned}
            \frac{d}{dt}{V}(\myvar{y}) 
            & \leq -k_0 \myvar{\sigma}_1^\top \Gamma_1 \myvar{\sigma}_1 + \sum_{i\in \myset{I}} \myvar{\sigma}_i^\top \Gamma_i \myvar{d}_i^{\mathbf{I}_i,\prime} \\
            & \leq  -k_0 \lambda_{\min}( ({A}_1^{\top,\mathbf{I}_1} {A}_1^{\top,\mathbf{I}_1,\top})^{-1}) \\
         & \hspace{2cm} + \sum_{i\in \myset{I}} \|  \myvar{\sigma}_i \| \| \Gamma_i \| \| \myvar{d}_i^{\mathbf{I}_i,\prime}\|
    \end{aligned}
    \end{equation}

    for almost all $t$. The first inequality comes from the positive semi-definiteness of $\Gamma_i$, and the second inequality follows from the positive definiteness of $({A}_1^{\top,\mathbf{I}_1} {A}_1^{\top,\mathbf{I}_1,\top})^{-1}$ and $\| \myvar{\sigma}_1^{\mathbf{I}_1} \|\geq 1$. Note that $ \myvar{\sigma}_i \in \mathbb{R}^M, [\myvar{\sigma}_i]_m = [L\sign{L \myvar{c}^m}]_i $, then $\| \myvar{\sigma}_i\| \leq \sqrt{M}\| L\| $, $\| \myvar{d}_i^{\mathbf{I}_i,\prime}\| = \| d_i^{\mathbf{I}_i} \|\leq D$. Thus, from \eqref{eq:k_0}, 
\begin{equation}
    \frac{d}{dt}V(\myvar{y})  \leq - k_0 \underline{\lambda} + ND\sqrt{M}\bar{\lambda} \| L\| \leq -\epsilon
\end{equation}
\end{enumerate}

Based on these discussions, we know $\frac{d}{dt}V \leq -\epsilon$ for almost all $t$ as long as $L\myvar{c}^m\neq \myvar{0}$ for some $m \in \myset{M}$ for arbitrary systems $\Sigma_{\mathbf{I}_i, i\in \myset{I}}$. Hence,  leveraging the theory of common Lyapunov functions,
it gives a trivial upper estimate of the convergence time as $\sum_{m\in \myset{M}}\| L\myvar{c}^m(0)\|_1/\epsilon$.   From Proposition~\ref{prop:constraint_satisfaction}, we further conclude that the coupling constraints are fulfilled for all time.
\end{proof}

Theorem~\ref{thm:time-varying_cbf_qp_properties} shows that the solution optimal to the slowly time-varying quadratic program \eqref{eq:CBF_problem} is obtained in a violation-free, finite-time, and distributed manner. These properties are favorable for implementing the CBF-induced quadratic programs as in \eqref{eq:CBF_problem} in a safe,  (point-wise) optimal, and distributed way.

One could easily recognize that the update law in \eqref{eq:adaptive law finite time} and the analysis in Theorem~\ref{thm:time-varying_cbf_qp_properties} are similar to those in \cite{tan2021distributed}, where a simple case with a single coupling constraint is considered. We note that this extension is not trivial due to the nonlinear relations between $\myvar{c}_i $ and $\myvar{y}$ in \eqref{eq:local_C}. Compared to \cite{tan2021distributed}, here we distinguish active constraints from inactive ones in order to obtain an analytical relation between changes in $\myvar{c}_i $ and $\myvar{y}$, and show that the Lyapunov function monotonically decreases for any active constraints, i.e., it is a common Lyapunov function.

\section{Simulations} \label{sec:simulation}
In this section, we demonstrate the effectiveness of the proposed algorithms in a static optimization problem and a multi-agent safe coordination application. All the simulations are conducted in Matlab R2022b on an Intel Core i7-1365U Windows laptop, and the primal-dual pairs are solved by the built-in \texttt{quadprog} function with default parameters. We apply forward Euler integration for simulating the continuous-time dynamics with a sampling time $0.01$s. 

\subsection{A static online optimization problem}
In this subsection, we consider an optimization problem that involves $N = 9$ agents indexed in $\myset{I}=\{1,2,..,9\}$ with the communication graph shown in Fig.~\ref{fig:graph}. Here each agent $i$ has $6$ decision variables $\myvar{x}_i = (x_{i,1},...,x_{i,6})$, where $(x_{i,1},x_{i,2},x_{i,3}) $ represents the demands of $3$  products, and $(x_{i,4},x_{i,5},x_{i,6}) $ the supplies of $3$ resources.  The optimization problem is given by
\begin{equation} \label{eq:static_program_example}
    \begin{split}
        &\min_{\myvar{x}_1,\dots,\myvar{x}_9} \sum_{i\in\myset{I}} (\myvar{x}_i^\top Q \myvar{x}_i  + \textstyle \sum_{j= 4,5,6} h_{i,j} {x}_{i,j} ) \\
        & {s.t.} \quad \left\{
        \begin{array}{cc}
          \sum_{i\in\myset{I}} 2x_{i,2}+x_{i,3} - x_{i,4}  \leq 0, \\
            \sum_{i\in\myset{I}} 2x_{i,1}+x_{i,3} - x_{i,5}  \leq 0, \\
             \sum_{i\in\myset{I}} x_{i,1}+x_{i,2} - x_{i,6}  \leq 0, \\
             x_{i,1}\geq h_{i,1}, x_{i,2}\geq h_{i,2}, x_{i,1}\geq h_{i,3}, \forall i\in \myset{I}
        \end{array}
        \right.
    \end{split}
\end{equation}
where $h_{i,j} = \textup{ceil}(10\sin(ij)+20), i\in \myset{I}, j = 1,2,...,6, Q$ is symmetric and $\myvar{x}_i^\top Q \myvar{x}_i = (2x_{i,2}+x_{i,3} - x_{i,4} )^2 + (2x_{i,1}+x_{i,3} - x_{i,5})^2 + (x_{i,1}+x_{i,2} - x_{i,6})^2 $. The first three inequality constraints represent that the total supply for the resources needs to exceed the total demand  of producing $3$ different products; the last three inequality constraints represent the local production output requirements. In the cost function, we penalize the derivation between the local supply and local demand, which can be interpreted as the cost of transporting or storing  the deficient or surplus resources, and take into account the local cost of obtaining those resources. We assume that the local production output requirement $h_{i,j}, j = 1,2,3$ and the local cost coefficient $h_{i,j},j= 4,5,6$ are only known to agent $i$.

It is clear that the local cost function $f_i(\myvar{x}_i) = \myvar{x}_i^\top Q \myvar{x}_i  + \textstyle \sum_{j= 4,5,6} h_{i,j} {x}_{i,j} $ is convex but not strongly convex, and the local coefficient matrix per the definition in \eqref{eq:local_problem} $$A_i^\top  = \begin{psmallmatrix}
     0   &   2 &    1  &   -1   &  0  &   0 \\
     2   &   0 &     1 &     0 &    -1 &      0 \\
     1  &    1   &   0  &    0  &    0  &  -1 \\
     -1   &  0   &  0  &   0  &   0   &  0 \\
     0  &   -1   &  0   &  0   &  0   &  0 \\
     0   &  0  &   -1   &  0  &   0   &  0
\end{psmallmatrix}$$ is of full rank. This fulfills Assumption~\ref{ass:local_slater} on Slater's condition.

\begin{figure}
    \centering
    \includegraphics{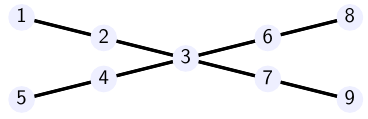}
    \caption{Communication graph}
    \label{fig:graph}
\end{figure}

\begin{figure*}[ht]
	\centering
	\begin{subfigure}[t]{0.31\linewidth}
		\includegraphics[width=\linewidth]{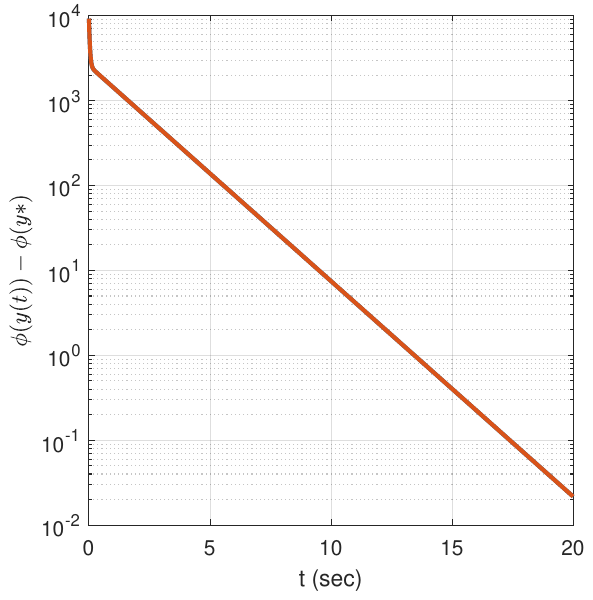}
		\caption{   Time evolution of the total cost minus the optimal cost. } 
	\end{subfigure}
	\begin{subfigure}[t]{0.31\linewidth}
		\centering\includegraphics[width=\linewidth]{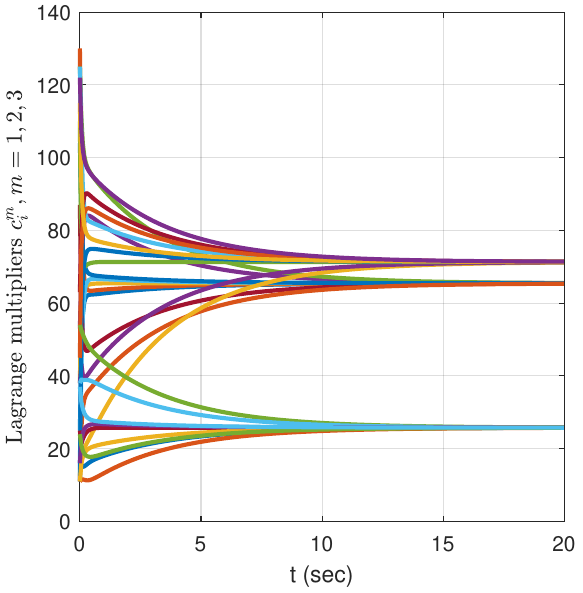}
		\caption{  Time evolution of the local Lagrange multipliers corresponding to the three coupling constraints}
	\end{subfigure}
	\begin{subfigure}[t]{0.31\linewidth}
		\centering\includegraphics[width=\linewidth]{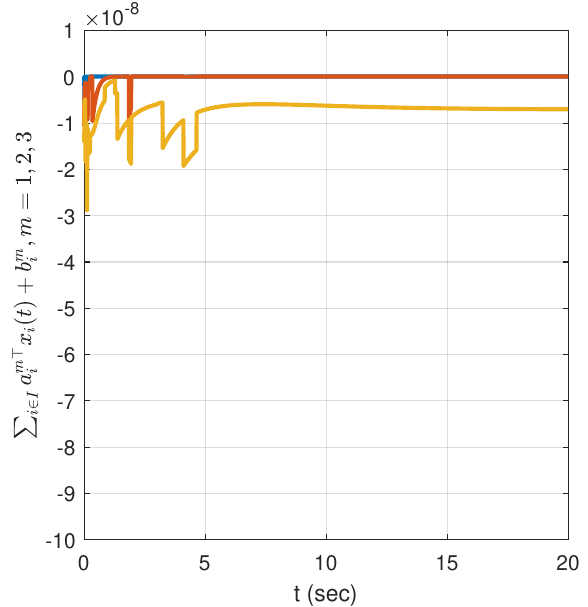}
		\caption{ Time evolution of the value of the first three coupling constraints }
	\end{subfigure}
	\setlength{\belowcaptionskip}{-10pt}
	\caption{  Numerical results involving $9$ agents solving a static optimization problem. }
	\label{fig:case1_result}

    \end{figure*}
    
From the structure of the problem, since  the last three constraints only involve its own decision variable, the algorithms for the sparse case apply here and that the related $ \myvar{ y }^m$ equals zero. Our proposed approach introduces $27$ auxiliary scalar variables, with each node needing to communicate $6$ scalar variables with its neighbors. In comparison, the primal decomposition approach in~ \cite{notarnicola2019constraint} requires $48$ auxiliary scalar variables and $24$ scalar variables to be communicated by node~$3$. By choosing $y_i^m(0) = 0, \forall i\in \myset{I},m\in \myset{M}$, $k_0 = 1$ with sampling-time $ 0.01$s, we obtain the results shown in Fig.~\ref{fig:case1_result}. One observes that 1) the total cost error is monotonically decreasing; 2) the local Lagrange multipliers corresponding to each coupling constraint converge to a common value; and 3) the three coupling constraints are always fulfilled during the solution iterations. This violation-free behavior could be desirable in this scenario because one does not need to wait for the convergence of the optimization algorithm before applying the solutions, and the intermediate solution, though suboptimal, still guarantees that the supply-demand requirement is respected.

\begin{figure*}[ht]
	\centering
	\begin{subfigure}[t]{0.48\linewidth}
		\includegraphics[width=\linewidth]{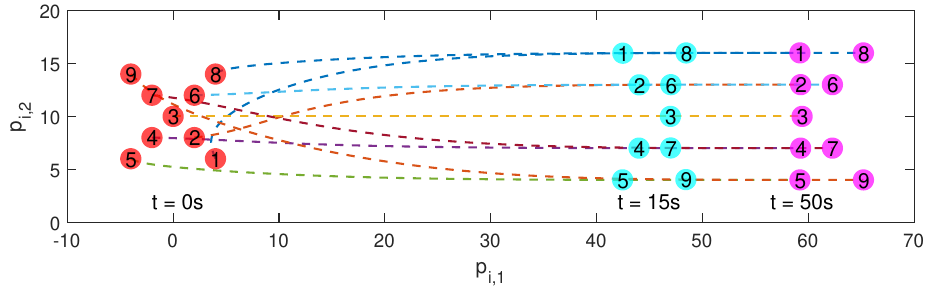}
		\caption{   Nominal case } 
	\end{subfigure}
	\begin{subfigure}[t]{0.48\linewidth}
		\centering\includegraphics[width=\linewidth]{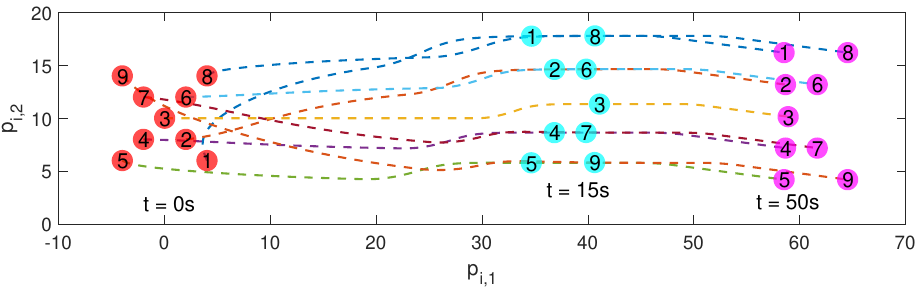}
		\caption{  Centralized Case}
	\end{subfigure}
	\begin{subfigure}[t]{0.48\linewidth}
		\centering\includegraphics[width=\linewidth]{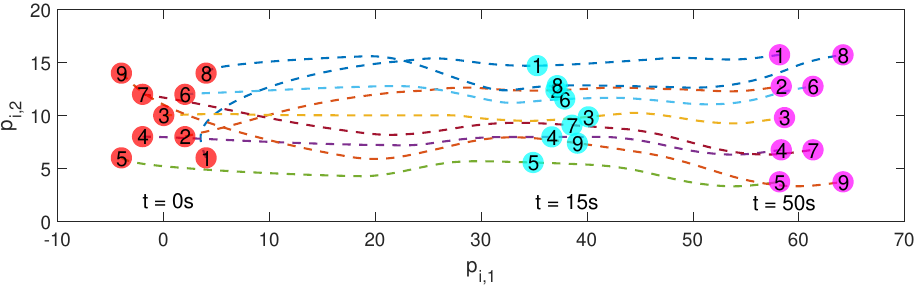}
		\caption{ Distributed case }
	\end{subfigure}
        \begin{subfigure}[t]{0.48\linewidth}
		\centering\includegraphics[width=\linewidth]{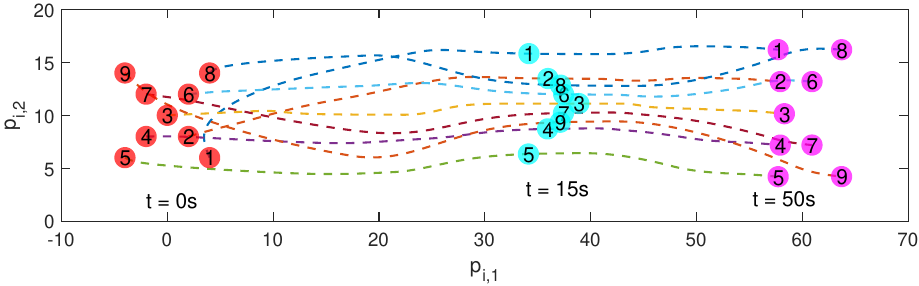}
		\caption{ Naive distributed case }
	\end{subfigure}
	\setlength{\belowcaptionskip}{-10pt}
	\caption{  Safe distributed control for coordinating $9$ agents. }
	\label{fig:case2_result}

    \end{figure*}

\begin{figure*}[ht]
	\centering
	\begin{subfigure}[t]{0.24\linewidth}
		\includegraphics[width=\linewidth]{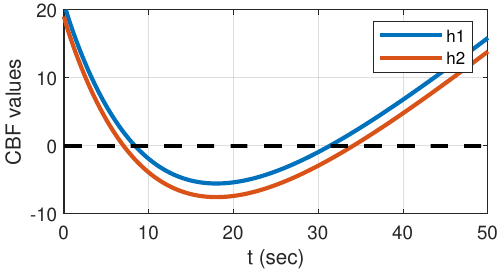}
		\caption{   Nominal case } 
	\end{subfigure}
	\begin{subfigure}[t]{0.24\linewidth}
		\centering\includegraphics[width=\linewidth]{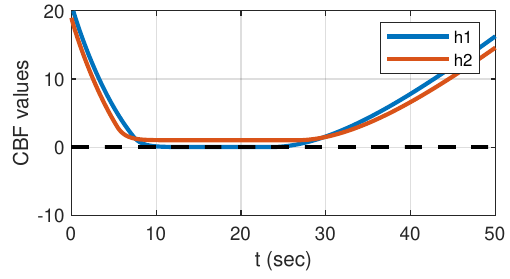}
		\caption{  Centralized case}
	\end{subfigure}
	\begin{subfigure}[t]{0.24\linewidth}
		\centering\includegraphics[width=\linewidth]{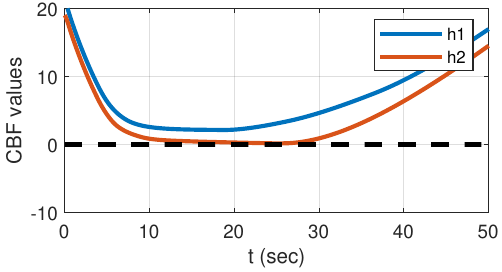}
		\caption{ Distributed case }
	\end{subfigure}
        \begin{subfigure}[t]{0.24\linewidth}
		\centering\includegraphics[width=\linewidth]{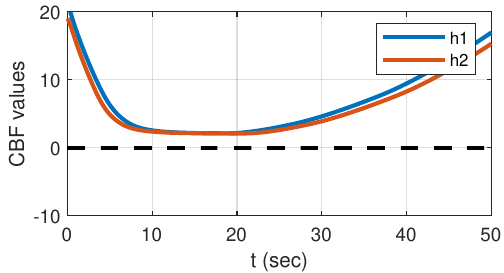}
		\caption{ Naive distributed case }
	\end{subfigure}
	\setlength{\belowcaptionskip}{-10pt}
	\caption{  Time evolution of the CBF values when coordinating $9$ agents. }
	\label{fig:case2_cbf}

    \end{figure*}

    \begin{figure}
        \centering
        \includegraphics[width=0.8\linewidth]{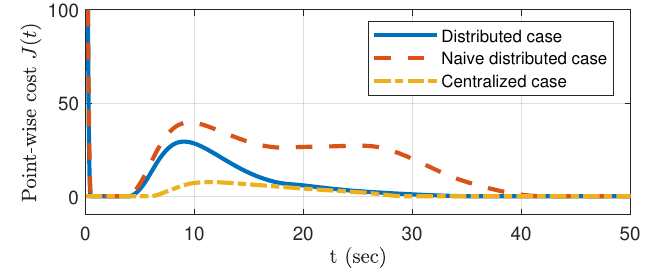}
        \caption{Time evolution of the  cost $\sum_{i\in \myset{I}}  \frac{1}{2}\| \myvar{x}_i -\myvar{x}_{nom,i} \|^2 $ along the state trajectory when the proposed algorithm is applied. }
        \label{fig:case2_cost}
    \end{figure}

\subsection{An application to safe distributed control for MAS}
In the second example, we show one  application of the proposed algorithm for safe distributed control of a multi-agent system. Denote by $\myvar{p}_i = (p_{i,1}, p_{i,2})\in \mathbb{R}^2$ the state  and $\myvar{p}_{i,d} \in \mathbb{R}^2$ the desired state of agent $i$. We assume that the dynamics of each agent is given by $\dot{\myvar{p}}_i = \myvar{x}_i$, where $\myvar{x}_i$ is the velocity command to agent $i$.  The coordination task is to transform the formation of the multi-agent system  from the initial formation $\myvar{p}_{\text{initial}} =  [\myvar{p}_{1}(0); \myvar{p}_{2}(0);...;\myvar{p}_{9}(0)] =  (4, 6, 2, 8, 0, 10, -2, 8, -4, 6, 2, 12, -2, 12, 4, 14, -4, 14)$ (shown in red in Fig.~\ref{fig:case2_result}) to the target formation $\myvar{p}_{\text{final}} = [\myvar{p}_{1,d}; \myvar{p}_{2,d};...;\myvar{p}_{9,d}]=(60, 16, 60, 13, 60, 10, 60, 7, $ $  60, 4, 63, 13, 63, 7, 66, 16, 66, 4)$. The communication graph is in Fig.~\ref{fig:graph}. Assume that agent $i, i\in \myset{I}$ has access to the desired relative position $\myvar{p}_{ij,d} = \myvar{p}_{j,d} - \myvar{p}_{i,d}, j\in N_i $ and its absolute position $\myvar{p}_i$, and agent $3$, referred to as the leader, knows its target position $\myvar{p}_{3,d}$. An intuitive  distributed protocol for fulfilling this task is given by
 \begin{equation*}
     \myvar{x}_{nom,i}(\myvar{p}_i, \{\myvar{p}_j\}_{j\in N_i}) = \sum_{j\in N_i } \myvar{p}_{ij,d} - \myvar{p}_{ij}  + \delta_i (\myvar{p}_{i,d} - \myvar{p}_{i}),
 \end{equation*}
 where $\myvar{p}_{ij} = \myvar{p}_{j} - \myvar{p}_{i} $, 
 $\delta_i = 1$ if agent $i$ is the leader and $\delta_i = 0$ otherwise.

Suppose that the multi-agent team needs to satisfy $2$ coupling state constraints during the formation transition, given by 
\begin{subequations}
 \begin{align}
        \hspace{-3mm}  h_{1}(t,\myvar{p}_1,...,\myvar{p}_9) &= 30 + t - 0.1\sum_{i\in \myset{I}} (p_{i,1} + p_{i,2}) \geq 0 , \label{eq:mas_cbf_1}\\
         \hspace{-3mm}  h_{2}(t,\myvar{p}_1,...,\myvar{p}_9) &= 10 + t - 0.1\sum_{i\in \myset{I}} ( p_{i,1} - p_{i,2}) \geq 0. \label{eq:mas_cbf_2}
    \end{align}
\end{subequations}
These two inequalities impose time-varying upper bounds on the collective summation of all coordinates as in \eqref{eq:mas_cbf_1} and the collective difference between the first and the second coordinates as in \eqref{eq:mas_cbf_2}, respectively. When viewing the two coordinates of $\myvar{p}_i$ as two different local sources and the overall task as transitioning from one resource allocation scheme to another, one may interpret the first constraint as a gradual relaxation on the source limit, and the second one as the gradual relaxation on the difference limit between the resources.

Following the control barrier function-induced control design procedures \cite{Ames2017}, the centralized safety-assuring controller is given by 
\begin{equation} \label{eq:mas_cbf_centralized}
    \begin{aligned}
        & \min_{\myvar{x}_1,\myvar{x}_2, ..., \myvar{x}_9} \sum_{i\in \myset{I}}  \frac{1}{2}\| \myvar{x}_i -\myvar{x}_{nom,i} \|^2  \\ 
        \textup{ s. t. } 
        & \sum_{i\in \myset{I}}[ 0.1 \ 0.1] \myvar{x}_i -  h_1(t,\myvar{p}_1,...,\myvar{p}_9) -1\leq 0, \\
        & \sum_{i\in \myset{I}}[ 0.1 \ -0.1] \myvar{x}_i -  h_2(t,\myvar{p}_1,...,\myvar{p}_9) -1\leq 0.
    \end{aligned}
\end{equation}
Here we choose the extended class $\mathcal{K}$ function to be an identity function for simplicity. It is clear that \eqref{eq:mas_cbf_centralized} is a special case of \eqref{eq:CBF_problem} with $2$ coupling inequality constraints.  Denote the corresponding auxiliary variables by $\myvar{y}^1, \myvar{y}^2$. Based on previous analysis, the local quadratic program-based controller is given by 
\begin{equation} \label{eq:mas_cbf_local}
    \begin{aligned}
        & \hspace{15mm}\min_{\myvar{x}_i} \frac{1}{2}\| \myvar{x}_i -\myvar{x}_{nom,i} \|^2  \\ 
        \textup{ s. t. } & [ 0.1 \ 0.1] \myvar{x}_i  + p_{i,1} + p_{i,2} + \myvar{l}_i\myvar{y}^{1} -  31/9 \leq 0,    \\
        & [ 0.1 \ -0.1] \myvar{x}_i   +p_{i,1} - p_{i,2} + \myvar{l}_i\myvar{y}^{2} -  11/9 \leq 0.
    \end{aligned}
\end{equation}
where $\myvar{l}_i$ is the $i$-th row of the Laplacian matrix.  These auxiliary variables are updated locally according to \eqref{eq:adaptive law finite time}. In this case we know $A_i^\top = 0.1\begin{psmallmatrix}
    1 & 1\\
    1 & -1
    \end{psmallmatrix}$ is of full column rank and the LICQ of the local optimization problem is always fulfilled. Moreover, recall that $\myvar{x}_{nom,i}(\myvar{p}(t))$ and $\myvar{y}^m(t), m= 1,2$ are locally Lipschitz in time, we know the control input from this local optimization problem is also locally Lipschitz in  time  \cite[Theorem 3.1]{hager1979lipschitz}. Together with the violation-free property from Theorem~\ref{thm:time-varying_cbf_qp_properties}, our proposed algorithm guarantees system safety.

In Fig.~\ref{fig:case2_result} we present the trajectories of four different cases: 1) the nominal case, where the controller $\myvar{x}_{nom,i}$ is applied, 2) the centralized case, where we assume a central node exists that computes the safe controller from \eqref{eq:mas_cbf_centralized} and sends the control signal to each agent; 3) the distributed case, where our proposed safe controller in \eqref{eq:mas_cbf_local} is implemented; 4) the naive distributed case, where the controller from \eqref{eq:mas_cbf_local} is used with all $\myvar{y}^{1},\myvar{y}^2$ zero. The last case is often used in CBF literature when system designers focus on guaranteeing safety and would sacrifice point-wise optimality.

In the simulation, we choose $k_0 = 1$ and replace $\text{sign}(\cdot)$ with a saturation function $\text{sat}(v) =  1, \text{if } v >0.1;
    10v, \text{if } -0.1\leq v \leq0.1;
    -1, \text{if } v <-0.1 $ for eliminating the numerical chattering. One observes that the formation task is fulfilled by all four control schemes. In Fig.~\ref{fig:case2_cbf}, it becomes clear that the two coupling state constraints are violated in the  nominal case, and respected in the other three cases. Moreover, the distributed case with the proposed algorithm shows to be less conservative compared to the naive distributed scheme as the trajectory approaches closer to the safety boundary. To better quantify the performance of the proposed algorithm, we further compare the point-wise cost $\sum_{i\in \myset{I}}  \frac{1}{2}\| \myvar{x}_i -\myvar{x}_{nom,i} \|^2 $ in Fig.~\ref{fig:case2_cost} along the state trajectory when our proposed algorithm is applied, i.e., the trajectory in Fig.~\ref{fig:case2_result}(c). At any time $t$, one could view the curves corresponding to the naive distributed case and  the centralized case in Fig.~\ref{fig:case2_cost} as the initial cost and the optimal cost of the frozen-time optimization problem, respectively. Compared to the naive implementation, our proposed algorithm achieves safe distributed control with much less point-wise cost.

One key prerequisite that can restrict applicability of the proposed approach is Assumption~\ref{ass:local_slater}. Even though we provide sufficient conditions for it to hold in dense and sparse cases, it can still be limiting. We will explore how to online monitor local feasibility and adapt the local parameter $y_i$ with feasibility guarantees in our future work. 

\section{Conclusion}

In this work, we propose a continuous-time violation-free, distributed algorithm for constraint-coupled optimization problems and demonstrate its applicability for control barrier function (CBF)-induced safe distributed control design. To decompose the global constraint coupled optimization problem into equivalent local ones, auxiliary variables are introduced,  whose differences indicate the allocation of the constraints among all agents.   Moreover, the  subgradient information with respect to the auxiliary variables can be obtained using local information, which leads to  a subgradient algorithm-like update law  design with asymptotic optimality and all-time constraint satisfaction guarantee.
For safe distributed control design, we further strengthen our algorithms to achieve finite-time convergence to cope with slowly time-varying parameters, leading to a safety-assured and point-wise optimal controller.
Two numerical examples are provided to illustrate the favorable properties of the proposed algorithms: a static resource allocation optimization example and a multi-agent coordination example with two coupling state constraints.

\section*{Appendix}
\begin{proof}[Proof of Proposition~\ref{prop:equivalent_problem}]
     To prove Point 1),  given any feasible solution $(\myvar{x}^{\prime},\myvar{y}^{\prime})$ to \eqref{eq:reformulated_problem}, for an arbitrary $m$-th constraint  in \eqref{eq:centralized_problem},
    we know $\myvar{1}^\top (A^{m \top} \myvar{x}^\prime + L\myvar{y}^{\prime m} + \myvar{b}^m ) = \sum_{i\in\myset{I}} {\myvar{a}_i^{m \top}}\myvar{x}_i^\prime  + b_i^m  \leq 0  $ thanks to the fact that $\myvar{1}^\top L = \myvar{0}$ for $L$ being the Laplacian matrix of any connected and undirected graph $\mathcal{G}$ \cite{mesbahi2010graph}. Thus $\myvar{x}^{\prime}$ also satisfies the $m$-th constraint in \eqref{eq:centralized_problem}. 
    
    For Point 2), given any feasible solution $\myvar{x}^\prime$ to \eqref{eq:centralized_problem}, for an arbitrary $m$-th constraint in \eqref{eq:centralized_problem}, define $\myvar{v}^{m} = A^{m \top}\myvar{x}^\prime + \myvar{b}^{m}$, $\myvar{w}^{m} = \frac{\myvar{1}_N^\top \myvar{v}^{m}}{N}  \myvar{1}_N$. Thus, $\myvar{w}^{m} \le \myvar{0}$ is obtained from the fact that $\myvar{w}^{m} $ has the same value in each entry and $\myvar{1}_N^\top \myvar{w}^{m} = \myvar{1}_N^\top \myvar{v}^{m} = \sum_{i\in\myset{I}} {\myvar{a}_i^{m \top}}\myvar{x}_i^\prime  + b_i^m  \leq 0 $. We further know that there exists a $\myvar{y}^m\in \mathbb{R}^N$ such that $L\myvar{y}^m + \myvar{v}^m = \myvar{w}^m$ thanks to the fact that $  \text{Range}(L) = \{ \myvar{z}\in \mathbb{R}^N: \myvar{1}_N^\top \myvar{z} = 0\}$ for a connected undirected graph \cite{mesbahi2010graph}. Thus, such a $(\myvar{x}^{\prime},\myvar{y}^1,  ..., \myvar{y}^M)$ is also feasible to the problem in \eqref{eq:reformulated_problem}.   
    
    Point 3) is obvious in view of the two problems. Based on the previous analysis, \eqref{eq:centralized_problem} and \eqref{eq:reformulated_problem} share the same set of feasible solutions with respect to $\myvar{x}$. Noting that they also share the same cost function, we conclude that the two problems are equivalent. 
\end{proof}

\begin{proof}[Proof of  Proposition~\ref{prop:local analytical solution}]
    For problem \eqref{eq:local_problem_qp}, we consider its Lagrangian
    \begin{equation*}
    \mathcal{L}(\myvar{x}_i, {\myvar{c}_i}) =  \frac{1}{2}\lVert \myvar{x}_i - \myvar{x}_{nom,i} \rVert^2 + \myvar{c}_i^\top\left(A_i^\top  \myvar{x}_i + (I_M\otimes \myvar{l}_i) \myvar{y}+ \myvar{b}_i  \right) 
\end{equation*}
where $\myvar{c}_i = (c_i^1, c_i^2,...,c_i^M) \in \mathbb{R}^{M}$ is the Lagrange multiplier. Since Slater's condition holds and the  problem \eqref{eq:local_problem_qp} is a linearly constrained convex optimization problem, $\myvar{x}_i$ is an optimal solution if and only if $\myvar{c}_i$ exists and satisfies the following 
Karush–Kuhn–Tucker (KKT) conditions \cite[Theorem 10.6]{beck2014introduction}
\begin{subequations}
    \begin{align}
        &\myvar{x}_i - \myvar{x}_{nom,i} + A_i \myvar{c}_i = \myvar{0} \label{eq:optimality_for_u} \\
        &  A_i^\top  \myvar{x}_i + (I_M\otimes \myvar{l}_i ) \myvar{y}+ \myvar{b}_i  \leq \myvar{0} \label{eq:primal_feasi} \\
        & \myvar{c}_i \geq 0  \label{eq:dual_feasi}\\
        & \myvar{c}_i \circ \left(A_i^\top  \myvar{x}_i + (I_M\otimes \myvar{l}_i ) \myvar{y}+ \myvar{b}_i  \right) = \myvar{0}   \label{eq:slackness}
    \end{align}
\end{subequations}
where $\circ$ denotes Hadamard Product, i.e., pointwise multiplication. To solve the KKT system, we obtain from \eqref{eq:optimality_for_u} that
\begin{equation}\label{eq:relation_u_unom}
    \myvar{x}_i =  \myvar{x}_{nom,i} - A_i\myvar{c}_i =  \myvar{x}_{nom,i} - \sum_{m=1}^M c_i^m{\myvar{a}_i^m}  
\end{equation}
Substituting it to \eqref{eq:primal_feasi} and \eqref{eq:slackness}, we know
\begin{subequations}
    \begin{align}
        & A_i^\top \myvar{x}_{nom,i} - A_i^\top A_i \myvar{c}_i + (I_M\otimes \myvar{l}_i)\myvar{y} + \myvar{b}_i \leq 0 \\
        & \myvar{c}_i\circ(A_i^\top \myvar{x}_{nom,i} - A_i^\top A_i \myvar{c}_i + (I_M\otimes \myvar{l}_i)\myvar{y} + \myvar{b}_i) = 0
    \end{align}
\end{subequations}

Next, we consider two disjoint cases for the $m$-th coupled inequality. Before proceeding, we denote $\mathcal{C}_m:=\{1, \dots, m-1, m+1, \dots, M \}$. We assume that the Lagrangian variables corresponding to the index set $\mathcal{C}_m$ are known, i.e., $c_i^p,  p\in \mathcal{C}_m$ are known.  Define $\myvar{x}^{m \prime}_{nom,i} = \myvar{x}_{nom,i} - \sum_{p\in \mathcal{C}_m} c_i^p \myvar{a}_i^p, \myvar{b}_i^{\prime} = (I_M\otimes \myvar{l}_i)\myvar{y} + \myvar{b}_i$. Then $\myvar{x}_i 
 = \myvar{x}^{m \prime}_{nom,i} - c_i^m \myvar{a}_i^m$ from \eqref{eq:relation_u_unom}. Define $\Lambda_i = \text{Diag}({A}_i^\top{A}_i) \in \mathbb{R}^{M\times M} $.  Consider the following cases:

\begin{enumerate}

    \item  $[{A_i^m}^\top \myvar{x}^{m \prime}_{nom,i}  + \myvar{b}_i^\prime ]_m \leq  0$. It is clear from \eqref{eq:slackness} that 
    \begin{equation} \label{eq:c_inactive}
    c_i^m = 0.
\end{equation}
    Then $c_i^m = 0$ together with $c_i^p, p \in \mathcal{C}_m$ is the Lagrange multiplier.

 \item $[{A_i^m}^\top \myvar{x}^{m \prime}_{nom,i}  + \myvar{b}_i^\prime ]_m > 0$. Suppose $c_i^m = 0$ in this case, then $\myvar{x}_i =\myvar{x}^{m \prime}_{nom,i} $ yet it fails the constraint ${A_i^m}^\top \myvar{x}_{i}  + \myvar{b}_i^\prime \leq \myvar{0}$.  By contradiction, we have $c_i^m>0$.  Therefore, to satisfy \eqref{eq:slackness}, there must hold
    \begin{equation}  \label{eq:case2_original_condition}
         [{A_i}^\top \myvar{x}^{m \prime}_{nom,i}  + \myvar{b}_i^\prime   -  c_i^m A_i^{\top}\myvar{a}_i^m]_m = 0. 
    \end{equation} 
   
Thus,
\begin{equation}
   \myvar{a}_i^{m \top} \myvar{a}_i^m c_i^m = [{A_i^m}^\top \myvar{x}^{m \prime}_{nom,i}  + \myvar{b}_i^\prime ]_m 
\end{equation}
Recall $ [\Lambda_{i}]_{m,m}   =  \myvar{a}_i^{m \top} \myvar{a}_i^m$, $\myvar{x}^{m \prime}_{nom,i} = \myvar{u}_{nom,i} - {A}_i \myvar{c}_i + c_i^m \myvar{a}_i^m$. Thus
\begin{equation} \label{eq:c_active2}
\begin{aligned}
        \hspace{-3mm} [\Lambda_i]_{m,m}  c_i^m  
     = [{A}_i^\top (\myvar{u}_{nom,i} - {A}_i \myvar{c}_i) + \myvar{b}_i^\prime  + \Lambda_i \myvar{c}_i ]_m 
\end{aligned}
\end{equation}

    \end{enumerate}

Summarizing the two cases in  \eqref{eq:c_inactive} and \eqref{eq:c_active2}, we obtain $c_i^m, m\in \myset{M}$ satisfies 
\begin{equation*}
      [\Lambda_i]_{m,m}  c_i^m  
     = \max \left( [{A}_i^\top (\myvar{x}_{nom,i} - {A}_i \myvar{c}_i) + \myvar{b}_i^\prime + \Lambda_i \myvar{c}_i ]_m, 0\right).
\end{equation*}
Stacking it together, we conclude that 
\begin{equation} \label{eq:algebraic condition on c_i}
\begin{aligned}
     \hspace{-3mm}   \Lambda_i \myvar{c}_i &  = \max \left({A}_i^\top \myvar{x}_{nom,i} +\myvar{b}_i^\prime  + (\Lambda_i - {A}_i^\top{A}_i) \myvar{c}_i,\myvar{0}\right)
\end{aligned}
\end{equation}  

Up to now we have shown that if $(\myvar{x}_i, \myvar{c}_i)$ 
 satisfies the KKT condition, then $\myvar{c}_i$ satisfies \eqref{eq:algebraic condition on c_i}. In the following we show that whenever $\myvar{c}_i$ solves \eqref{eq:algebraic condition on c_i}, the pair $(\myvar{x}_i, \myvar{c}_i)$ with $\myvar{x}_i = \myvar{x}_{nom,i} - A_i \myvar{c}_i$ fulfills the KKT conditions. One verifies that \eqref{eq:optimality_for_u} and \eqref{eq:dual_feasi} are trivially satisfied. What remains to show is \eqref{eq:primal_feasi} and \eqref{eq:slackness}. Denote by $c_i^p, p\in \mathcal{C}_m$ as part of the solution to \eqref{eq:algebraic condition on c_i}. Now $c_i^p, p\in \mathcal{C}_m$ are not assumed to be the correct multipliers but will be proved so later. Recall $\myvar{x}_{nom,i}^{m \prime} = \myvar{x}_{nom,i} - \sum_{p\in \mathcal{C}_m} c_i^p \myvar{a}_i^p$. If $[{A_i}^\top \myvar{x}^{m \prime}_{nom,i}  + \myvar{b}_i^\prime]_m \leq  0$, we know \eqref{eq:algebraic condition on c_i} dictates $c_i^m = 0$.  If $ [A_i^\top \myvar{x}^{m\prime}_{nom,i}  + \myvar{b}_i^\prime  ]_m>0 $, then \eqref{eq:algebraic condition on c_i} dictates $\myvar{a}_i^{m \top} \myvar{a}_i^m c_i^m = [{A}_i^{\top} \myvar{x}_{nom,i}^{m \prime} + \myvar{b}_i^\prime]_m    $.  In both cases,  the parts in \eqref{eq:primal_feasi} and \eqref{eq:slackness} corresponding to the $m$-th constraint are fulfilled. This analysis holds for all $m\in \myset{M}$. This concludes the proof.
\end{proof}

\bibliographystyle{IEEEtran}
\bibliography{IEEEabrv,references}

\begin{thebibliography}{10}
\providecommand{\url}[1]{#1}
\csname url@samestyle\endcsname
\providecommand{\newblock}{\relax}
\providecommand{\bibinfo}[2]{#2}
\providecommand{\BIBentrySTDinterwordspacing}{\spaceskip=0pt\relax}
\providecommand{\BIBentryALTinterwordstretchfactor}{4}
\providecommand{\BIBentryALTinterwordspacing}{\spaceskip=\fontdimen2\font plus
\BIBentryALTinterwordstretchfactor\fontdimen3\font minus
  \fontdimen4\font\relax}
\providecommand{\BIBforeignlanguage}[2]{{%
\expandafter\ifx\csname l@#1\endcsname\relax
\typeout{** WARNING: IEEEtran.bst: No hyphenation pattern has been}%
\typeout{** loaded for the language `#1'. Using the pattern for}%
\typeout{** the default language instead.}%
\else
\language=\csname l@#1\endcsname
\fi
#2}}
\providecommand{\BIBdecl}{\relax}
\BIBdecl

\bibitem{notarnicola2019constraint}
I.~Notarnicola and G.~Notarstefano, ``Constraint-coupled distributed
  optimization: a relaxation and duality approach,'' \emph{IEEE Transactions on
  Control of Network Systems}, vol.~7, no.~1, pp. 483--492, 2019.

\bibitem{nedic2018distributed}
A.~Nedi{\'c} and J.~Liu, ``Distributed optimization for control,'' \emph{Annual
  Review of Control, Robotics, and Autonomous Systems}, vol.~1, pp. 77--103,
  2018.

\bibitem{yang2019survey}
T.~Yang, X.~Yi, J.~Wu, Y.~Yuan, D.~Wu, Z.~Meng, Y.~Hong, H.~Wang, Z.~Lin, and
  K.~H. Johansson, ``A survey of distributed optimization,'' \emph{Annual
  Reviews in Control}, vol.~47, pp. 278--305, 2019.

\bibitem{falsone2017dual}
A.~Falsone, K.~Margellos, S.~Garatti, and M.~Prandini, ``Dual decomposition for
  multi-agent distributed optimization with coupling constraints,''
  \emph{Automatica}, vol.~84, pp. 149--158, 2017.

\bibitem{simonetto2016primal}
A.~Simonetto and H.~Jamali-Rad, ``Primal recovery from consensus-based dual
  decomposition for distributed convex optimization,'' \emph{Journal of
  Optimization Theory and Applications}, vol. 168, pp. 172--197, 2016.

\bibitem{mateos2016distributed}
D.~Mateos-N{\'u}nez and J.~Cort{\'e}s, ``Distributed saddle-point subgradient
  algorithms with {L}aplacian averaging,'' \emph{IEEE Transactions on Automatic
  Control}, vol.~62, no.~6, pp. 2720--2735, 2016.

\bibitem{li2020distributed}
X.~Li, G.~Feng, and L.~Xie, ``Distributed proximal algorithms for multiagent
  optimization with coupled inequality constraints,'' \emph{IEEE Transactions
  on Automatic Control}, vol.~66, no.~3, pp. 1223--1230, 2020.

\bibitem{nedic2009approximate}
A.~Nedi{\'c} and A.~Ozdaglar, ``Approximate primal solutions and rate analysis
  for dual subgradient methods,'' \emph{SIAM Journal on Optimization}, vol.~19,
  no.~4, pp. 1757--1780, 2009.

\bibitem{su2021distributed}
Y.~Su, Q.~Wang, and C.~Sun, ``Distributed primal-dual method for convex
  optimization with coupled constraints,'' \emph{IEEE Transactions on Signal
  Processing}, vol.~70, pp. 523--535, 2021.

\bibitem{falsone2023augmented}
A.~Falsone and M.~Prandini, ``Augmented lagrangian tracking for distributed
  optimization with equality and inequality coupling constraints,''
  \emph{Automatica}, vol. 157, p. 111269, 2023.

\bibitem{lasdon2002optimization}
L.~S. Lasdon, \emph{Optimization theory for large systems}.\hskip 1em plus
  0.5em minus 0.4em\relax Courier Corporation, 2002.

\bibitem{wang2023distributed}
H.~Wang, A.~Papachristodoulou, and K.~Margellos, ``Distributed control design
  and safety verification for multi-agent systems,'' \emph{arXiv preprint
  arXiv:2303.12610}, 2023.

\bibitem{liu2015second}
Q.~Liu and J.~Wang, ``A second-order multi-agent network for bound-constrained
  distributed optimization,'' \emph{IEEE Transactions on Automatic Control},
  vol.~60, no.~12, pp. 3310--3315, 2015.

\bibitem{wu2022distributed}
X.~Wu, S.~Magn{\'u}sson, and M.~Johansson, ``Distributed safe resource
  allocation using barrier functions,'' \emph{Automatica}, vol. 153, p. 111051,
  2023.

\bibitem{abeynanda2023primal}
H.~Abeynanda, C.~Weeraddana, G.~Lanel, and C.~Fischione, ``On the primal
  feasibility in dual decomposition methods under additive and bounded
  errors,'' \emph{IEEE Transactions on Signal Processing}, vol.~71, pp.
  655--669, 2023.

\bibitem{tan2021distributed}
X.~Tan and D.~V. Dimarogonas, ``Distributed implementation of control barrier
  functions for multi-agent systems,'' \emph{IEEE Control Systems Letters},
  vol.~6, pp. 1879--1884, 2021.

\bibitem{mestres2023distributed}
P.~Mestres and J.~Cortes, ``Distributed and anytime algorithm for network
  optimization problems with seperable structure,'' in \emph{2023 IEEE
  Conference on Decision and Control (CDC)}.\hskip 1em plus 0.5em minus
  0.4em\relax IEEE, 2023, pp. 5457--5462.

\bibitem{cherukuri2015distributed}
A.~Cherukuri and J.~Cort{\'e}s, ``Distributed generator coordination for
  initialization and anytime optimization in economic dispatch,'' \emph{IEEE
  Transactions on Control of Network Systems}, vol.~2, no.~3, pp. 226--237,
  2015.

\bibitem{liu2024achieving}
C.~Liu, X.~Tan, X.~Wu, D.~V. Dimarogonas, and K.~H. Johansson, ``Achieving
  violation-free distributed optimization under coupling constraints,''
  \emph{arXiv preprint arXiv:2404.07609}, 2024.

\bibitem{mesbahi2010graph}
M.~Mesbahi and M.~Egerstedt, \emph{Graph theoretic methods in multiagent
  networks}.\hskip 1em plus 0.5em minus 0.4em\relax Princeton University Press,
  2010.

\bibitem{clarke2008nonsmooth}
F.~H. Clarke, Y.~S. Ledyaev, R.~J. Stern, and P.~R. Wolenski, \emph{Nonsmooth
  analysis and control theory}.\hskip 1em plus 0.5em minus 0.4em\relax Springer
  Science \& Business Media, 2008, vol. 178.

\bibitem{bertsekas1995nonlinear}
D.~P. Bertsekas, \emph{Nonlinear programming}.\hskip 1em plus 0.5em minus
  0.4em\relax Belmont: Athena Scientific,, 1995.

\bibitem{boyd2004convex}
S.~Boyd and L.~Vandenberghe, \emph{Convex optimization}.\hskip 1em plus 0.5em
  minus 0.4em\relax Cambridge university press, 2004.

\bibitem{cortes2008discontinuous}
J.~Cortes, ``Discontinuous dynamical systems,'' \emph{IEEE Control systems
  magazine}, vol.~28, no.~3, pp. 36--73, 2008.

\bibitem{robinson1980strongly}
S.~M. Robinson, ``Strongly regular generalized equations,'' \emph{Mathematics
  of Operations Research}, vol.~5, no.~1, pp. 43--62, 1980.

\bibitem{mestres2023robinson}
P.~Mestres, A.~Allibhoy, and J.~Cort{\'e}s, ``Robinson's counterexample and
  regularity properties of optimization-based controllers,'' \emph{arXiv
  preprint arXiv:2311.13167}, 2023.

\bibitem{florenzano2001finite}
M.~Florenzano and C.~Le~Van, \emph{Finite dimensional convexity and
  optimization}.\hskip 1em plus 0.5em minus 0.4em\relax Springer Science \&
  Business Media, 2001, vol.~13.

\bibitem{wachsmuth2013licq}
G.~Wachsmuth, ``On {LICQ} and the uniqueness of {L}agrange multipliers,''
  \emph{Operations Research Letters}, vol.~41, no.~1, pp. 78--80, 2013.

\bibitem{Wang2017a}
L.~Wang, A.~Ames, and M.~Egerstedt, ``Safety barrier certificates for
  collisions-free multirobot systems,'' \emph{IEEE Transactions on Robotics},
  vol.~33, no.~3, pp. 661--674, 2017.

\bibitem{fernandez2023distributed}
V.~N. Fernandez-Ayala, X.~Tan, and D.~V. Dimarogonas, ``Distributed barrier
  function-enabled human-in-the-loop control for multi-robot systems,'' in
  \emph{2023 IEEE International Conference on Robotics and Automation
  (ICRA)}.\hskip 1em plus 0.5em minus 0.4em\relax IEEE, 2023, pp. 7706--7712.

\bibitem{mestres2024distributed}
P.~Mestres, C.~Nieto-Granda, and J.~Cort{\'e}s, ``Distributed safe navigation
  of multi-agent systems using control barrier function-based controllers,''
  \emph{IEEE Robotics and Automation Letters}, vol.~9, no.~7, pp. 6760 -- 6767,
  2024.

\bibitem{Ames2019control}
A.~D. Ames, S.~Coogan, M.~Egerstedt, G.~Notomista, K.~Sreenath, and P.~Tabuada,
  ``Control barrier functions: Theory and applications,'' in \emph{Proc.
  European Control Conf.}, 2019, pp. 3420--3431.

\bibitem{tan2022compatibility}
X.~Tan and D.~V. Dimarogonas, ``Compatibility checking of multiple control
  barrier functions for input constrained systems,'' in \emph{2022 IEEE 61st
  Conference on Decision and Control (CDC)}.\hskip 1em plus 0.5em minus
  0.4em\relax IEEE, 2022, pp. 939--944.

\bibitem{santilli2020distributed}
M.~Santilli, G.~Oliva, and A.~Gasparri, ``Distributed finite-time algorithm for
  a class of quadratic optimization problems with time-varying linear
  constraints,'' in \emph{2020 59th IEEE Conference on Decision and Control
  (CDC)}.\hskip 1em plus 0.5em minus 0.4em\relax IEEE, 2020, pp. 4380--4386.

\bibitem{cortes2006finite}
J.~Cort{\'e}s, ``Finite-time convergent gradient flows with applications to
  network consensus,'' \emph{Automatica}, vol.~42, no.~11, pp. 1993--2000,
  2006.

\bibitem{liberzon2003switching}
D.~Liberzon, \emph{Switching in systems and control}.\hskip 1em plus 0.5em
  minus 0.4em\relax Springer, 2003, vol. 190.

\bibitem{franceschelli2014finite}
M.~Franceschelli, A.~Pisano, A.~Giua, and E.~Usai, ``Finite-time consensus with
  disturbance rejection by discontinuous local interactions in directed
  graphs,'' \emph{IEEE transactions on Automatic Control}, vol.~60, no.~4, pp.
  1133--1138, 2014.

\bibitem{Ames2017}
A.~D. Ames, X.~Xu, J.~W. Grizzle, and P.~Tabuada, ``Control barrier function
  based quadratic programs for safety critical systems,'' \emph{IEEE
  Transactions on Automatic Control}, vol.~62, no.~8, pp. 3861--3876, 2016.

\bibitem{hager1979lipschitz}
W.~W. Hager, ``Lipschitz continuity for constrained processes,'' \emph{SIAM
  Journal on Control and Optimization}, vol.~17, no.~3, pp. 321--338, 1979.

\bibitem{beck2014introduction}
A.~Beck, \emph{Introduction to nonlinear optimization: Theory, algorithms, and
  applications with MATLAB}.\hskip 1em plus 0.5em minus 0.4em\relax SIAM, 2014.

\end{thebibliography}

\end{document}